\documentclass[11pt]{article}
\usepackage[margin=1in]{geometry}
\usepackage{amsmath, amsfonts, amssymb, amsthm}
\usepackage{bm}
\usepackage{comment}
\usepackage{enumitem}
\usepackage{algorithm2e}
\usepackage{xcolor}
\usepackage{makecell}
\usepackage{hyperref}
\hypersetup{colorlinks,breaklinks,
    citecolor=[RGB]{153,204,153},
    linkcolor=[RGB]{102,153,204},
    urlcolor=[RGB]{115,140,191}
}
\SetKwInput{Initialization}{Initialization}
\SetKwInput{Bounding}{Bounding}
\SetKwInput{BoundingL}{Bounding the Left Node}
\SetKwInput{BoundingR}{Bounding the Right Node}
\SetKwInput{Branching}{Branching}
\SetKwInput{Pruning}{Pruning}
\SetKwInput{Termination}{Termination}
\SetKwInput{Output}{Output}
\SetKwInput{WarmStart}{Warm-Start}
\SetKwInput{return}{return}
\theoremstyle{plain}
\newtheorem{theorem}{Theorem}
\newtheorem{proposition}{Proposition}
\newtheorem{lemma}{Lemma}

\theoremstyle{definition}
\newtheorem{assumption}{Assumption}

\theoremstyle{plain}
\newtheorem{remark}{Remark}
\providecommand{\noopsort}[1]{}
\newcommand{\one}{{\mathbf 1}} 
\newcommand{\RR}{{\mathbb R}} 
\newcommand{\II}{{\mathbb I}}

\newcommand{\norm}[1]{\left\lVert#1\right\rVert}
\newcommand{\abs}[1]{\left\lvert#1\right\rvert}

\newcommand{\tpose}[1]{#1^{\top}}

\newcommand{\cF}{{\mathcal F}}
\newcommand{\cG}{{\mathcal G}}

\newcommand{\cP}{{\mathcal P}}

\newcommand{\ccr}{{\mathfrak r}}

\newcommand{\cY}{{\mathcal Y}}
\newcommand{\cX}{{\mathcal X}}
\newcommand{\cZ}{{\mathcal Z}}
\usepackage[sort&compress]{natbib}
 \bibpunct[, ]{[}{]}{,}{n}{}{,}%

\title{Sparsity Regularized and Robust Mean Variance Portfolio Selection Under Ellipsoidal Uncertainty}
\author{Deniz Akkaya\thanks{Department of Industrial Engineering, Bilkent University, Ankara, Türkiye}
\and Emre Can Yayla\footnotemark[1]
\and Buse \c{S}en\thanks{Risk Analytics and Optimization Chair, EPFL, Lausanne, Switzerland}
\and Mustafa Ç. Pınar\footnotemark[1]}
\date{}

\begin{document}
\maketitle

\begin{abstract}
We investigate mean–variance portfolio selection with an $\ell_0$-penalty to promote sparsity in asset allocations. Uncertainty in the mean return vector is incorporated through an ellipsoidal uncertainty set, yielding a robust sparse optimization framework. We characterize the structure of both local and global minimizers and exploit these properties in the risk minimization and return maximization formulations. Building on this structural insight, we develop a branch-and-bound algorithm tailored to the resulting robust sparse portfolio problems, together with a new pruning rule that can discard exponentially many candidate portfolios in a single step. Extensive computational experiments on real market data, together with comparisons against a mixed-integer second-order cone programming solver, demonstrate the effectiveness and competitiveness of the proposed approach.

\medskip
\noindent\textbf{Keywords:} Mean-variance portfolio, robust optimization, regularization, sparsity, $\ell_0$-norm, Branch-and-Bound

\medskip
\noindent\textbf{Mathematics Subject Classification (MSC 2020):} 90C11, 91G10, 90C17, 90C26, 90C57
\end{abstract}

\section{Introduction}\label{sec:Intro}
Mean-variance portfolio selection, originally introduced by \cite{Markowitz1952}, forms the cornerstone of modern portfolio theory. In this framework, an investor determines portfolio weights by balancing expected return against risk measured through variance, thereby tracing the efficient frontier. Despite its conceptual elegance and widespread adoption, the classical mean-variance model is known to be highly sensitive to estimation errors in the input parameters, particularly in the expected return vector. As emphasized in the literature, small perturbations in estimated returns may lead to large and economically unintuitive changes in optimal portfolios, a phenomenon often attributed to the amplification of estimation noise \cite{Michaud1989}.

To mitigate this instability, robust optimization has emerged as a systematic approach for incorporating parameter uncertainty directly into optimization models \cite{BenTalNemirovski1998,BenTalNemirovski2009,ElGhaouiOustryLebret1998}. In robust portfolio optimization, uncertain parameters such as expected returns are assumed to lie within prescribed uncertainty sets, and portfolio decisions are made to perform well under the worst-case realization within these sets. A prominent modeling choice is the use of ellipsoidal uncertainty sets for the mean return vector, which lead to tractable reformulations and admit appealing statistical interpretations. In particular, \cite{GoldfarbIyengar2003} demonstrated that robust mean-variance portfolio problems with ellipsoidal uncertainty can be reformulated as convex optimization problems and that the resulting portfolios exhibit improved stability and out-of-sample performance. Subsequent contributions, including \cite{BertsimasSim2004}, explored alternative uncertainty structures and their computational implications. In addition, distributionally robust optimization has provided a broader framework for modeling uncertainty in financial decision making, where ambiguity in probability distributions is explicitly incorporated \cite{DelageYe2010,GohSim2010,WiesemannKuhnSim2014}. Among various robust formulations, we follow the analysis of robust portfolio models developed in \cite{Pinar2016} as a foundation for our study.

While robustness addresses estimation risk, practical portfolio construction typically involves additional structural considerations. In many real-world applications, investors restrict the number of assets held in a portfolio due to transaction costs, liquidity constraints, monitoring costs, or regulatory requirements. Such considerations naturally lead to sparse portfolio models, where the number of nonzero positions is explicitly controlled. Early work on cardinality-constrained portfolio optimization highlighted the computational challenges arising from these restrictions \cite{ChangMeadeBeasleySharaiha2000,LoboFazelBoyd2007,BonamiLejeune2009}. More recently, sparsity-inducing formulations and scalable algorithms have attracted considerable attention in the literature. Contributions such as \cite{Brodie2009,BertsimasCoryWright2022,WuSunGeZhaoZeng2024,KobayashiTakanoNakata2023} demonstrate the practical and computational advantages of sparse portfolios. A direct way to enforce sparsity is through cardinality constraints or $\ell_0$-regularization, where the $\ell_0$-term counts the number of selected assets. However, the inclusion of such a term leads to nonconvex and combinatorial optimization problems that are challenging to solve. In \cite{SenAkkayaPinar2025}, a comprehensive analysis of sparse portfolio optimization is presented, including structural properties of optimal portfolios and an efficient enumeration-based branching algorithm. In addition, recent studies continue to explore models that combine sparsity and robustness in portfolio construction \cite{ZhaoJiangYang2023,AnnunziataEtAl2025}.

Despite the extensive literature on robust portfolio optimization and the growing body of work on sparse portfolio selection, the integration of ellipsoidal mean uncertainty with exact $\ell_0$-regularization remains relatively unexplored. Most robust formulations focus primarily on mitigating estimation risk but do not explicitly control portfolio cardinality. Conversely, sparse portfolio models often rely on deterministic parameter estimates and do not explicitly account for uncertainty in expected returns. A unified treatment that simultaneously addresses estimation risk and enforces exact sparsity leads to a challenging class of robust mixed-integer quadratic optimization problems. Motivating this integration, preliminary out-of-sample tests on real equity market data show that adding ellipsoidal robustness to a sparse mean-variance model can improve risk-adjusted performance over the purely sparse benchmark.

In this paper, we study sparse mean-variance portfolio selection under ellipsoidal uncertainty in the mean return vector. We consider both risk minimization and return maximization variants and incorporate an $\ell_0$-penalty to promote portfolios with a prescribed level of sparsity. The ellipsoidal uncertainty set is consistent with the robust optimization framework developed in \cite{Pinar2016,BenTalNemirovski2009,GoldfarbIyengar2003}, while the $\ell_0$-regularization follows the exact sparsity-inducing perspective studied in \cite{Nikolova2013,AkkayaPinar2020,AkkayaPinar2025} and recent sparse portfolio optimization research presented in \cite{SenAkkayaPinar2025}.

Our contributions are both theoretical and computational. On the theoretical side, we characterize the structure of local and global minimizers of the robust sparsity-penalized problem. We analyze how the interaction between the quadratic risk term, the worst-case mean adjustment induced by the ellipsoidal uncertainty set, and the discontinuous $\ell_0$-penalty determines the structure of optimal portfolios. These results extend structural analyses known for deterministic sparse mean-variance models to a robust setting and clarify the role of the robustness parameter in shaping portfolio composition.

On the computational side, we develop a tailored branch-and-bound framework that exploits problem-specific lower and upper bounds derived from the analytical structure of the model. These bounds enable effective pruning of candidate supports and provide a solid basis for warm-start heuristics. We conduct extensive computational experiments on real financial data and benchmark our approach against mixed-integer second-order conic formulations. The numerical results show that the proposed method is computationally effective, especially on larger instances, and often produces high-quality sparse robust portfolios with substantially reduced running times.

Overall, the paper contributes both structural analysis and algorithmic developments for robust sparse mean-variance portfolio optimization. The main contributions of this paper are as follows.
\begin{itemize}
\item We propose a robust sparse mean-variance portfolio model that integrates ellipsoidal uncertainty in expected returns with exact $\ell_0$-regularization, providing a unified treatment of robustness and sparsity.
\item We develop a structural analysis of the resulting nonconvex problem, including support-wise decomposition, identification of local minimizers and results on existence and properties of global minimizers.
\item We establish explicit bounds on optimal solutions and show that sparsity levels can be controlled via the regularization parameter through thresholding results.
\item We design a tailored branch-and-bound algorithm that turns these structural insights into bounding and warm-start strategies, and show on real data that it is computationally effective. A key ingredient is an additional pruning step that refines the enumeration scheme of~\cite{SenAkkayaPinar2025} for the nonrobust sparse portfolio problem: it can eliminate exponentially many subproblems in a single iteration and applies verbatim to that scheme.
\end{itemize}
We introduce the sparsity-regularized robust portfolio problem in two alternative formulations and present the necessary background and notation in Section~\ref{probdef}. Analytical results for both variants are presented in Sections~\ref{theory1} and \ref{theory2}, where we study structural properties of locally and globally optimal portfolios, including bounds on the number of nonzero entries and existence results. A branch-and-bound algorithm is proposed in Section~\ref{BBalgo}. Finally, we compare our algorithm with a state-of-the-art solver for mixed-integer second-order cone programming problems in Section~\ref{computation} and conclude in Section~\ref{conclusion}.
\section{Problem Definition and Notation}\label{probdef}
Let $\II_N=(\{1,\dots,N\},<)$ be the strictly ordered index set, where $<$ denotes the standard order. Any subset $\omega\subseteq\II_N$ inherits this property. We denote the all-ones vector by $\one$, the identity matrix by $\mathbf{I}$, and the zero vector by $\mathbf{0}$. We denote the $i^{th}$ column of a matrix $D$ by $d_i$. For $\omega\subseteq\II_N$, the following notation for subvectors and submatrices will be used:
\begin{align*}
	r_\omega:=(r[\omega[1]],\dots,r[\omega[\abs{\omega}]])\in\RR^{\abs{\omega}},
	\quad D_\omega:=\big((d_{\omega[1]})_\omega,\dots,(d_{\omega[\abs{\omega}]})_\omega\big) \in \RR^{\abs{\omega}\times\abs{\omega}}.
\end{align*}
We define the zero-padding operator $Z_\omega:\RR^{\abs{\omega}}\rightarrow\RR^N$ by
\begin{align*}
x=Z_\omega(x_\omega), \quad
x[i]=\begin{cases}
0,& i\notin\omega,\\
x_\omega[k],& \text{if } \omega[k]=i.
\end{cases}
\end{align*}
We introduce the indicator $\phi:\RR\rightarrow\{0,1\}$ defined by
\begin{align*}
\phi(t)=\begin{cases}
0, & t=0,\\
1, & t\neq 0,
\end{cases}
\quad \text{so that } \norm{x}_0:=\sum_{i\in\II_N} \phi(x[i])=\sum_{i\in\sigma(x)}\phi(x[i]).
\end{align*}
Here, $\norm{x}_0=\abs{\sigma(x)}$, where $\sigma(x)$ is the support of $x$ (indices of nonzero entries), and $\abs{\cdot}$ denotes cardinality.
For $x\in\RR^N$, the $\ell_p$ norm is defined for $1\leq p<\infty$ as $\norm{x}_p:=\left(\sum_{i\in\II_N} \abs{x[i]}^p\right)^{1/p}$, $\norm{x}_\infty=\max_{i\in\II_N}\{\abs{x[i]}\}$,  and the matrix induced norm by a positive definite matrix $D$ is $\norm{x}_D=\sqrt{\tpose{x}Dx}$. Given $\rho>0$, the open $\ell_p$ ball of radius $\rho$ centered at $x$ is $B_p(x,\rho):=\{y\in\RR^N:\norm{x-y}_p<\rho\}$. For a matrix $A\in\RR^{M\times N}$, the spectral norm is $\norm{A}_2=s_1(A)$, where $s_i(A)$ is the $i^{th}$ singular value in decreasing order. 

We consider a financial market consisting of $N$ risky assets and a single risk-free asset with deterministic period return $r_c$. The vector of expected returns of the risky assets, denoted by $r\in \RR^N$, is unknown, while their return covariance matrix $D\in\RR^{N\times N}$ is assumed to be known and positive definite. The initial wealth is normalized to one.
Uncertainty in the mean returns is modeled through the ellipsoidal uncertainty set
\begin{align*}
    U_{\hat{r}} := \left\{ r \in \RR^N : \norm{ r - \hat{r}}_{D^{-1}} \leq \gamma \right\},
\end{align*}
where $\hat{r}$ denotes the nominal estimate of the mean return vector and $\gamma>0$ is a tolerance parameter controlling the size of the uncertainty set. Let $\bar{r}$ denote a prescribed target return level.
The investor selects portfolio weights $x\in\RR^N$ for the risky assets and $x_c\in\RR$ for the risk-free asset so as to minimize portfolio variance while ensuring that the target return is achieved for all admissible realizations of the mean return vector. The resulting robust mean-variance portfolio optimization problem is formulated as
\begin{align*}
    \begin{array}{ll}
     \min  & \tpose{x} D x \\
     \text{s.t.}
& \tpose{\one} x + x_c = 1\\
& \tpose{r} x + r_c x_c \geq \bar r \quad \forall r \in U_{\hat{r}}\\
& (x,x_c) \in \RR^{N+1}.
    \end{array}
\end{align*}
Under ellipsoidal uncertainty in the mean return vector, the semi-infinite robust constraint can be reduced to a single deterministic inequality. In particular, requiring that the portfolio achieves the target return for all admissible realizations of $r$,
\begin{align*}
    \tpose{r} x + r_c(1-\tpose{\one} x) \geq \bar{r}
\quad \forall r \in U_{\hat{r}},
\end{align*}
is equivalent to enforcing that the nominal expected return, penalized by the worst-case deviation induced by the uncertainty set, exceeds the target level. This yields the deterministic robust counterpart
\begin{align*}
\tpose{\hat{r}} x + r_c(1-\tpose{\one} x) - \gamma \norm{x}_D \geq \bar{r},
\end{align*}
where the term $\gamma\norm{x}_D$ arises from minimizing the linear form $\tpose{r}x$ over the ellipsoidal set $U_{\hat{r}}$. Consequently, the original semi-infinite robust mean-variance problem admits the equivalent finite-dimensional formulation
\begin{align}
\label{eq:robustPortfolio}
\min_{x\in\mathbb{R}^N} \; \tpose{x}D x
\quad \text{s.t.} \quad
\tpose{\hat{r}} x + r_c(1-\tpose{\one} x) - \gamma \norm{x}_D \geq \bar r.
\end{align}
This equivalence follows from standard results in robust optimization, where linear constraints subject to ellipsoidal uncertainty admit exact deterministic robust counterparts involving norm penalties \cite{GoldfarbIyengar2003,BenTalNemirovski2009}.
Let $\beta>0$ be a sparsity-inducing penalty parameter, and define the estimated \emph{excess return vector} by
$\ccr := \hat{r} - r_c \one$ together with the \emph{excess target return} $\bar{\ccr} := \bar{r} - r_c$.
We consider the following \emph{sparse robust mean-variance portfolio selection} problem:
\begin{align}
\tag{$\cP^1$}\label{prob:cprmvp}
\min_{x\in\RR^N} \; \cF_\beta(x):= \tpose{x} D x + \beta \norm{x}_0
\quad \text{s.t.} \quad
\tpose{\ccr} x - \gamma \norm{x}_D \ge \bar{\ccr}.
\end{align}
An alternative robust formulation to the problem presented in (\ref{eq:robustPortfolio}) is obtained by reversing the roles of risk and return. Instead of minimizing variance subject to a robust return requirement, one may fix an admissible risk level and maximize the worst-case portfolio return. For a suitably chosen parameter $T>0$, consider
\begin{align*}
\max_{x\in\RR^N} \;
\min_{r\in U_{\hat{r}}} \Bigl\{ \tpose{r} x + (1-\tpose{\one} x) r_c \Bigr\}
\quad \text{s.t.} \quad
\tpose{x} D x \leq T^2 .
\end{align*}
This formulation represents a robust counterpart of the classical mean-variance problem with a variance budget, and it provides an alternative scalarization of the risk-return trade-off.
Exploiting the ellipsoidal structure of the uncertainty set $U_{\hat{r}}$, the inner minimization over $r$ admits a closed-form solution, yielding the deterministic equivalent
\begin{align*}
\max_{x\in\RR^N} \;
\tpose{\hat{r}} x + (1-\tpose{\one} x) r_c - \gamma \norm{x}_D
\quad \text{s.t.} \quad
\tpose{x} D x \leq T^2.
\end{align*}
Since $\tpose{\hat{r}} x + (1-\tpose{\one} x) r_c=\tpose{\ccr}x+r_c$, the objective can be written in terms of excess returns as
\begin{align*}
\max_{x\in\RR^N} \;
\tpose{\ccr} x + r_c - \gamma \norm{x}_D
\quad \text{s.t.} \quad
\norm{x}_D \leq T ,
\end{align*}
and the constant $r_c$ can be dropped without affecting the set of optimal solutions.

While this model captures the same robustness considerations as problem~\eqref{eq:robustPortfolio}, the two are not equivalent. The original formulation enforces the target return as a hard robust constraint and minimizes risk accordingly, whereas the present model fixes the risk level and optimizes the worst-case return. Consequently, the two problems generally yield different optimal solutions, coinciding only for specific choices of $T$ that recover the risk level induced by the optimal solution of problem~\eqref{eq:robustPortfolio}.

Using the same sparsity-inducing penalty parameter $\beta$ as before, we introduce an alternative formulation aimed at promoting sparser solutions to the return maximization problem. Specifically, we consider the following optimization problem:
\begin{align}
\tag{$\cP^2$}\label{prob:cprmvp2}
\min_{x\in\RR^N} \;\cG_\beta(x):=
\gamma \norm{x}_D - \tpose{\ccr} x + \beta\norm{x}_0
\quad \text{s.t.} \quad
\norm{x}_D \leq T .
\end{align}
In the following sections, we develop the theoretical foundations for both formulations. Our analysis covers local optimality conditions, rigorous bounds on the nonzero components of global minimizers, asymptotic results concerning the existence of global minimizers, and principled guidelines for selecting the parameter $\beta$ to achieve prescribed sparsity levels.

\section{Analysis of \texorpdfstring{\eqref{prob:cprmvp}}{P1}}\label{theory1}
We begin this section by introducing an assumption that rules out the degenerate solution $x=\mathbf{0}$, in which all wealth is invested in the risk-free asset.
\begin{assumption}\label{assum:targetwealth}
The target return $\bar{r}$ strictly exceeds the period return $r_c$ of the risk-free asset.
\end{assumption}
We impose the following assumption to ensure that each asset's nominal excess return is sufficiently large relative to its risk contribution and the level of mean uncertainty, thereby guaranteeing feasibility of the associated subproblems.
\begin{assumption}\label{assump:feas}
The uncertainty radius $\gamma$ satisfies
\begin{align*}
\min_{i\in\II_N}\frac{\abs{\ccr[i]}}{\sqrt{d_i[i]}}>\gamma,
\end{align*}
where $\sqrt{d_i[i]}$ denotes the square root of $i^{th}$ diagonal entry of $D$.
\end{assumption}
Assumption~\ref{assump:feas} requires the magnitude of each asset's nominal
excess return relative to its volatility to exceed the uncertainty radius
$\gamma$. Since short positions are permitted, this ensures that every
singleton-support subproblem, and hence every nonempty support-restricted
subproblem, is feasible.
For ease of notation, define the feasible set of~\eqref{prob:cprmvp} as
\begin{align}
	\cX=\left\{x\in\RR^N:\tpose{\ccr}x-\gamma \norm{x}_D \geq \bar{\ccr}\right\}.\label{eq:cX}
\end{align}
We also define $H := \sqrt{\tpose{\ccr} D^{-1}\ccr}$, which represents the maximal achievable Sharpe ratio in the market and coincides with the slope of the capital market line.
\begin{remark}\label{rem:whyworks}
Note that for $x\in\RR^N$ and $\omega\supseteq\sigma(x)$, we have
$\tpose{x}Dx=\tpose{x_\omega}D_\omega x_\omega$.
This property motivates searching for local minimizers by restricting to supports.
\end{remark}
For $\omega\subseteq\II_N$, define $K_\omega=\{x\in\RR^N: x[i]=0,\ \forall i\in\omega^c\}$. We study the following subproblem to characterize local minimizers of $\eqref{prob:cprmvp}$:
\begin{align}\tag{$\cP^1_\omega$}\label{prob:mvp-omega}
	\min_{x\in \cX\cap K_\omega}\tpose{x}Dx.
\end{align}
Since $D$ is positive definite, the objective is strictly convex, and hence the problem has a unique solution whenever feasible.
\subsection{Minimizers of \texorpdfstring{\eqref{prob:mvp-omega}}{P1w}}
We define the restricted feasibility set for a fixed $\omega\subseteq\II_N$:
\begin{align*}
	\cX_\omega&=\left\{u\in\RR^{\abs{\omega}}:\tpose{\ccr_\omega}u-\gamma \norm{u}_{D_\omega}\geq \bar{\ccr}\right\}.
\end{align*}
Using these sets and the zero-padding operator, we define the equivalent problem to \eqref{prob:mvp-omega} with a convex feasible region
\begin{align}\tag{$\cZ\cP^1_\omega$}\label{prob:zpmvp}
	\min_{u\in \cX_\omega}\tpose{u}D_\omega u, \quad \abs{\omega}\geq 1.
\end{align}
\begin{remark}\label{rem:stillposdiag}
Assumption~\ref{assump:feas} is inherited by all subproblems: for any nonempty $\omega\subseteq\II_N$, $D_\omega$ is positive definite (as a principal submatrix of $D$) and $|\ccr_\omega[k]|>\gamma \sqrt{D_\omega[k,k]}$ holds for all $k\in\II_{\abs{\omega}}$, since the diagonal of $D_\omega$ consists of diagonal entries of $D$.
\end{remark}
The following lemma and remark are crucial in establishing that Assumption~\ref{assump:feas} ensures the existence of feasible solutions for every subproblem associated with a subset $\omega \subseteq \II_N$.
\begin{lemma}\label{lem:Fenchel}
    For any $v\in \RR^N$, we have $\tpose{v}D^{-1}v=\max_{u\in\RR^N}\{2\tpose{u}v-\tpose{u}Du\}$.
\end{lemma}
\begin{proof}Let $v\in\RR^N$ be fixed, and define $Q(u)=2\tpose{u}v-\tpose{u}Du$. Then
\begin{align*}
    Q(u)=-\tpose{(u-D^{-1}v)} D (u-D^{-1}v) + \tpose{v}D^{-1}v.
\end{align*}
Since $D$ is positive definite, the first term is non-positive for every $u\in\RR^N$, with equality if and only if $u=D^{-1}v$. Therefore the maximum is attained at $u=D^{-1}v$, and the desired identity follows.
\end{proof}
\begin{remark}\label{rem:feasall}
Let $u\in\cX_\omega$, and write $u=tv$ with $t>0$ and $\norm{v}_{D_\omega}=1$. Then $t(\gamma-\tpose{\ccr_\omega}v)+\bar{\ccr}\leq 0$.
Since $\bar{\ccr}>0$, feasibility requires $\gamma-\tpose{\ccr_\omega}v<0$. This holds whenever
\begin{align*}
\max_{\norm{v}_{D_\omega}=1}\tpose{\ccr_\omega}v=\sqrt{\tpose{\ccr_\omega}(D_\omega)^{-1}\ccr_\omega}=:H_\omega>\gamma,
\end{align*}
where $H_\omega=\sqrt{\tpose{\ccr_\omega}(D_\omega)^{-1}\ccr_\omega}$ is the Sharpe ratio of the subproblem defined by $\omega$. To guarantee feasibility across all supports, it suffices to ensure $\min_{\omega\neq\emptyset}H_\omega>\gamma$.
For any $\omega\neq \emptyset$ and $i\in\omega$ we use Lemma~\ref{lem:Fenchel} by taking $u=te_i\in\RR^{\abs{\omega}}$ and obtain:
\begin{align*}
    H_\omega=\sqrt{\tpose{\ccr_\omega}(D_\omega)^{-1}\ccr_\omega}\geq \sqrt{2t\ccr_\omega[i]-t^2d_i[i]}.
\end{align*}
Maximizing the right hand side over $t$ yields $t=\ccr_\omega[i]/d_i[i]$ and we have
\begin{align*}
    \sqrt{\tpose{\ccr_\omega}(D_\omega)^{-1}\ccr_\omega}\geq \frac{\abs{\ccr_\omega[i]}}{\sqrt{d_i[i]}}\Rightarrow  \sqrt{\tpose{\ccr_\omega}(D_\omega)^{-1}\ccr_\omega}\geq \max_{i\in\omega} \frac{\abs{\ccr_\omega[i]}}{\sqrt{d_i[i]}}.
\end{align*}
In fact, the minimum is attained on singletons:
\begin{align*}
\min_{\omega\neq\emptyset}H_\omega=H_{\min}:=\min_{i\in\II_N}\frac{\abs{\ccr[i]}}{\sqrt{d_{i}[i]}}>\gamma.
\end{align*}
Assumption~\ref{assump:feas} provides this bound. Thus, strict feasibility and convexity ensure a unique solution for~\eqref{prob:zpmvp}.
\end{remark}
\begin{proposition}\label{prop:xi}
For nonempty $\omega \subseteq \II_N$, the unique solution of \eqref{prob:zpmvp} is
\begin{align*}
\xi(\omega):=\left(\frac{\bar{\ccr}}{H_\omega(H_\omega-\gamma)}\right)(D_\omega)^{-1}\ccr_\omega.
\end{align*}
\end{proposition}
\begin{proof}
    For the proof, see \cite[Proposition~1]{Pinar2016}.
\end{proof}
\begin{remark}\label{rem:zpdef}
For $\omega\subseteq\II_N$ with $|\omega|\geq 1$, we write $\Xi(\omega)=Z_\omega(\xi(\omega))$ for its zero-padding to $\RR^N$.
\end{remark}
The next proposition gives the optimal multiplier of the subproblem constraint and shows that the constraint is active at optimality. The multiplier is used in the branching rule of the algorithm.
\begin{proposition}\label{prop:dual_solution}
For nonempty $\omega\subseteq \II_N$, the optimal Lagrange multiplier associated with the constraint of~\eqref{prob:zpmvp} is $\mu^* = \frac{2\bar{\ccr}}{(H_\omega - \gamma)^2}$.
Moreover, the constraint is active at the optimal solution $\hat u$, \textit{i.e.}, $\tpose{\ccr_\omega}\hat{u}-\gamma \norm{\hat{u}}_{D_\omega}= \bar{\ccr}$.
\end{proposition}
\begin{proof}
The proof parallels the argument in \cite[Proposition~1]{Pinar2016}. We also verify that the constraint $\tpose{\ccr_\omega}u-\gamma\norm{u}_{D_\omega}\ge\bar\ccr$ of~\eqref{prob:zpmvp} is active at optimality. Strict feasibility established above ensures that the KKT conditions apply.
Suppose that the constraint is inactive at $\hat u$. Complementary slackness
then gives $\mu^*=0$, and stationarity reduces to
$2D_\omega\hat u=0$. Since $D_\omega\succ0$, this implies $\hat u=0$,
which is infeasible because $\bar{\ccr}>0$ by
Assumption~\ref{assum:targetwealth}. Thus, the optimal solution $\hat{u}$ of \eqref{prob:zpmvp} satisfies $\tpose{\ccr_\omega}\hat{u}-\gamma \norm{\hat{u}}_{D_\omega}= \bar{\ccr}$.

Next, we show the dual multiplier in closed-form. By Proposition~\ref{prop:xi}, we have
\[
    D_\omega\hat u
    =\frac{\bar{\ccr}}{H_\omega(H_\omega-\gamma)}\ccr_\omega,
    \quad \text{and}\quad
    \frac{D_\omega\hat u}{\norm{\hat u}_{D_\omega}}
    =\frac{\ccr_\omega}{H_\omega}.
\]
Substituting these identities into the stationarity condition and rearranging the terms then yields  $\mu^*=\frac{2\bar{\ccr}}{(H_\omega-\gamma)^2}$.
 \end{proof}
The subproblem \eqref{prob:zpmvp} is posed in the reduced space $\RR^{\abs{\omega}}$. The following lemma confirms that solving it is equivalent to solving \eqref{prob:mvp-omega} in $\RR^N$.
\begin{lemma}\label{lemma5}
Problems \eqref{prob:zpmvp} and \eqref{prob:mvp-omega} are equivalent.
\end{lemma}
\begin{proof}
The map $Z_\omega:\cX_\omega\rightarrow\cX\cap K_\omega$ is a bijection. Moreover, for any $x_\omega\in\cX_\omega$, we have $\tpose{x_\omega}D_\omega x_\omega=\tpose{Z_\omega(x_\omega)}D Z_\omega(x_\omega)$. Hence the two formulations are equivalent.
\end{proof}
\begin{remark}
\label{rem:unique}
For $\omega\subseteq\II_N$ with $\abs{\omega}\geq 1$, the point $\Xi(\omega)\in\RR^N$ is the unique solution of \eqref{prob:mvp-omega}.
\end{remark}
\subsection{(Local) Minimizers of \texorpdfstring{\eqref{prob:cprmvp}}{P1}}
Since $\cF_\beta(x)=\tpose{x}Dx+
\beta\sum_{i\in\II_N}\phi(x[i])$, activating a zero component incurs an additional penalty $\beta$. The following proposition identifies a neighborhood in which this penalty cannot be offset by the change in the quadratic term of~\eqref{prob:cprmvp}.
\begin{proposition}\label{choiceofrho}
Let $\beta>0$ and $\hat{x}\in\cX$. Define $\hat{\sigma}=\sigma(\hat{x})$ and
\begin{align*}
\rho:=\min\left\{\min_{i\in\hat{\sigma}}|\hat{x}[i]|,\ \frac{\beta}{2(\|D\hat{x}\|_1+1)}\right\}.
\end{align*}
Then $\rho>0$, and:
\begin{enumerate}[label=(\roman*)]
\item If $y\in B_\infty(0,\rho)$, then
$\sum_{i\in\II_N}\phi(\hat{x}[i]+y[i])=\sum_{i\in\hat{\sigma}}\phi(\hat{x}[i])+\sum_{i\in\hat{\sigma}^c}\phi(y[i])$.
\item If $y\in B_\infty(0,\rho)\cap(\RR^N\setminus K_{\hat{\sigma}})$, then $\cF_\beta(\hat{x}+y)\geq \cF_\beta(\hat{x})$, with strict inequality whenever $\hat{\sigma}^c\neq\emptyset$.
\end{enumerate}
\end{proposition}
\begin{proof}
The proof follows identically from the argument in \cite[Lemma~2]{SenAkkayaPinar2025}.
\end{proof}
\begin{remark}
Proposition~\ref{choiceofrho} does not use feasibility of $\hat{x}+y$; in particular, it remains valid when $y$ is restricted to perturbations with $\hat{x}+y\in\cX$.
\end{remark}
The following two results establish a correspondence between local minimizers of~\eqref{prob:cprmvp} and global minimizers of~\eqref{prob:mvp-omega} for $\omega \subseteq \II_N$, in a manner analogous to Proposition~2 and Lemma~3 in~\cite{SenAkkayaPinar2025}. The proofs are therefore omitted.
\begin{proposition}\label{prop:localmin}
Let $\omega\subseteq\II_N$, $\omega\neq\emptyset$. For any $\beta>0$, the objective $\cF_\beta$ reaches a (local) minimum of \eqref{prob:cprmvp} at $\Xi(\omega)$, with $|\sigma(\Xi(\omega))|\geq 1$ and $\sigma(\Xi(\omega))\subseteq \omega$.
\end{proposition}
\begin{lemma}\label{lemma8}
Let $\beta>0$ and let $\hat{x}$ be a (local) minimizer of \eqref{prob:cprmvp}. Then $\hat{x}=\Xi(\sigma(\hat{x}))$.
\end{lemma}
Proposition~\ref{prop:localmin} and Lemma~\ref{lemma8} together show that the set of local minimizers of \eqref{prob:cprmvp} is precisely $\{\Xi(\omega) : \emptyset\neq\omega\subseteq\II_N\}$; in particular, every local minimizer is completely determined by its support.
\subsection{Global Minimizers of \texorpdfstring{\eqref{prob:cprmvp}}{P1}}
We begin by deriving a bounding box that contains all global minimizers of the problem. This box is subsequently employed for Big-M calibration and for tightening the feasible region.
\begin{proposition}\label{prop:globup}
   Let $\beta>0$ and suppose $\hat{x}$ is a global minimizer of~\eqref{prob:cprmvp}. Then we have
    \begin{align*}
        \norm{\hat{x}}_\infty\leq \sqrt{\frac{\eta}{s_N(D)}}<\infty,  \;  \text{ where } \eta=\min_{i\in\II_N}\left\{\frac{\bar{\ccr}^2d_i[i]}{\left(|\ccr[i]|-\gamma\sqrt{d_i[i]}\right)^2}\right\}.\end{align*}
\end{proposition}
\begin{proof}
Since $\hat{x}$ is a global minimizer of \eqref{prob:cprmvp}, for any $i\in\II_N$, we have
\begin{align*}
\tpose{\hat{x}}D\hat{x}+\beta\norm{\hat{x}}_0\leq \tpose{\Xi(\{i\})}D\Xi(\{i\})+\beta\|\Xi(\{i\})\|_0\leq\tpose{\Xi(\{i\})}D\Xi(\{i\})+\beta.
\end{align*}
Moreover, Assumption~\ref{assum:targetwealth} implies $\bar\ccr>0$, hence $0\notin\cX$ and therefore $\|\hat x\|_0\geq 1$, which results in $\|\hat x\|_D^2\leq \|\Xi(\{i\})\|_D^2$ for all $i\in\II_N$. By the definition of $\Xi$ on singleton supports,
\[
 \Xi(\{i\})=\frac{\bar{\ccr} \mathrm{sign}(\ccr[i])}{|\ccr[i]|-\gamma\sqrt{d_i[i]}} e_i,
\]
and therefore
\begin{align*}
\norm{\hat{x}}_D^2\leq \min_{i\in\II_N}\norm{\Xi(\{i\})}_D^2=\min_{i\in\II_N}\left\{\frac{\bar{\ccr}^2d_i[i]}{(\abs{\ccr[i]}-\gamma\sqrt{d_i[i]})^2}\right\}=\eta.
\end{align*}
For an upper bound on  the maximal entry, it is sufficient to rescale $\eta$ with the smallest singular value, that is $    \norm{\hat{x}}_\infty\leq \sqrt{\frac{\eta}{s_N(D)}}
$.
\end{proof}
The primary objective of Proposition~\ref{prop:globup} was to identify a bounding box that contains all globally optimal solutions. Nevertheless, deriving an upper bound on the variance is of independent interest. The following result establishes a strict separation of nonzero components from zero under the prescribed parameters. The resulting lower bound on the nonzero entries of global minimizers constitutes the central insight underlying the proposed warm-start heuristic.
\begin{theorem}
\label{thm:lower-bound}
    Let $\beta>0$ and suppose $\hat{x}$ is a global minimizer of~\eqref{prob:cprmvp}. Let also $\hat{\sigma}=\sigma(\hat{x})$ denote the support of $\hat{x}$, and we define for all $i\in\II_N$
    \begin{align*}
        \rho_i=\max_{\substack{s\in\{-1,1\}\\ j\in \II_N\setminus\{i\}}}\left\{\norm{e_i+s\frac{\abs{\ccr[i]}+\gamma\sqrt{d_i[i]}}{\abs{\ccr[j]}-\gamma\sqrt{d_j[j]}}e_j}_D\right\} \quad\text{ and } \quad\eta=\min_{i\in\II_N}\left\{\frac{\bar{\ccr}^2d_i[i]}{(\abs{\ccr[i]}-\gamma\sqrt{d_i[i]})^2}\right\} .
    \end{align*} Then for every $i\in\hat{\sigma}$, we have
    \begin{align*}
    \abs{\hat{x}[i]}\geq \min\left\{\frac{\sqrt{\eta+\beta}-\sqrt{\eta}}{\rho_i},\frac{\bar{\ccr}}{\abs{\ccr[i]}-\gamma\sqrt{d_i[i]}}\right\}.
\end{align*}
\end{theorem}
\begin{proof}
    Let $i\in\hat{\sigma}$. We consider two cases based on the cardinality of $\hat{\sigma}$.

    \textbf{Case 1:} Assume $\abs{\hat{\sigma}}\geq 2$, then there exists $j\in\hat{\sigma}$ such that $i\neq j$. We define $g_{ij}:\RR^N\rightarrow \RR^N$ as $g_{ij}(x)=x-x[i]e_i-x[j]e_j$ where $e_i$ and $e_j$ are canonical basis vectors of the specified index. We introduce the function $f(t_i,t_j)=\cF_\beta(g_{ij}(\hat{x})+t_ie_i+t_je_j)$. Finally, we define a feasibility function
    \begin{align*}
        h(t_i,t_j)=\tpose{\ccr}g_{ij}(\hat{x})+t_i\ccr[i]+t_j\ccr[j]-\gamma\norm{g_{ij}(\hat{x})+t_ie_i+t_je_j}_D-\bar{\ccr}.
    \end{align*}
   Since $\hat{x}$ is a global minimizer, Lemma~\ref{lemma8} gives $\hat{x}=\Xi(\hat{\sigma})$.  By Remark~\ref{rem:unique}, $\hat{x}$ is therefore the unique solution of \eqref{prob:mvp-omega} with $\omega=\hat{\sigma}$. Proposition~\ref{prop:dual_solution} then implies that the restricted constraint is active at $\hat{x}$. This, in turn,  implies that the main constraint is also active, yielding $h(\hat{x}[i],\hat{x}[j])= 0$. Precisely, we rewrite this equality as
    \begin{align*}
        h(\hat{x}[i],\hat{x}[j])=\tpose{\ccr}\hat{x}-\gamma\norm{\hat{x}}_D-\bar{\ccr}=0\Rightarrow \bar{\ccr}=\tpose{\ccr}\hat{x}-\gamma\norm{\hat{x}}_D.
    \end{align*}
    Now, we define a map by fixing a choice for $t_j$
    \begin{align*}
        h(t_i,t_j)&=\tpose{\ccr}g_{ij}(\hat{x})+t_i\ccr[i]+t_j\ccr[j]-\gamma\norm{g_{ij}(\hat{x})+t_ie_i+t_je_j}_D-\bar{\ccr}\\
        &=\gamma\left(\norm{\hat{x}}_D-\norm{g_{ij}(\hat{x})+t_ie_i+t_je_j}_D\right)+(t_i-\hat{x}[i])\ccr[i]+(t_j-\hat{x}[j])\ccr[j]\\&
        \geq-\gamma\left(\norm{(t_i-\hat{x}[i])e_i+(t_j-\hat{x}[j])e_j}_D\right)+(t_i-\hat{x}[i])\ccr[i]+(t_j-\hat{x}[j])\ccr[j]\\&
        \geq -\gamma\abs{t_i-\hat{x}[i]}\sqrt{d_i[i]}-\gamma\abs{t_j-\hat{x}[j]}\sqrt{d_j[j]}+(t_i-\hat{x}[i])\ccr[i]+(t_j-\hat{x}[j])\ccr[j] ,
    \end{align*}
    where the first inequality exploits  $\|u\|-\|v\|\geq -\|u-v\|$ with $u=\hat x$ and $v=g_{ij}(\hat x)+ t_i e_i+t_je_j$, and the second follows from the triangle inequality together with $\|e_i\|_D=\sqrt{d_i[i]}$. Introduce the shorthand  \[
     \delta_{ij}:=\frac{\abs{\ccr[i]}+\gamma\sqrt{d_i[i]}}{\abs{\ccr[j]}-\gamma\sqrt{d_j[j]}} >0,
    \]
    where positivity follows from Assumption~\ref{assump:feas}.
     Set $t_j=\mathrm{sign}(\ccr[j])\abs{t_i-\hat{x}[i]}\delta_{ij}+\hat{x}[j]$. Substituting this choice into the lower bound above yields
    \begin{align*}
        h(t_i,t_j)&\geq -\gamma\abs{t_i-\hat{x}[i]}\sqrt{d_i[i]}+(t_i-\hat{x}[i])\ccr[i]+\delta_{ij}\abs{t_i-\hat{x}[i]}(\abs{\ccr[j]}-\gamma\sqrt{d_j[j]})\\
        &\geq -(\abs{\ccr[i]}+\gamma\sqrt{d_i[i]})\abs{t_i-\hat{x}[i]}+\delta_{ij}\abs{t_i-\hat{x}[i]}(\abs{\ccr[j]}-\gamma\sqrt{d_j[j]})=0.
    \end{align*}
    Thus this choice ensures $h(t_i,t_j)\geq 0$. We may define a 1-dimensional restricted objective as
    \begin{align*}
        f(t)&=\norm{g_{ij}(\hat{x})+te_i+(\mathrm{sign}(\ccr[j])\abs{t-\hat{x}[i]}\delta_{ij}+\hat{x}[j])e_j}_D^2\\
        &\qquad +\beta \norm{g_{ij}(\hat{x})+te_i+(\mathrm{sign}(\ccr[j])\abs{t-\hat{x}[i]}\delta_{ij}+\hat{x}[j])e_j}_0\\
        &=\norm{\hat{x}+(t-\hat{x}[i])e_i+\mathrm{sign}(\ccr[j])\abs{t-\hat{x}[i]}\delta_{ij}e_j}_D^2\\
        &\qquad+\beta \norm{\hat{x}+(t-\hat{x}[i])e_i+\mathrm{sign}(\ccr[j])\abs{t-\hat{x}[i]}\delta_{ij}e_j}_0.
    \end{align*}
    By the choice of $t_j$ above, the point $\hat{x}-\hat{x}[i]e_i+\mathrm{sign}(\ccr[j])\abs{\hat{x}[i]}\delta_{ij}e_j$ corresponding to $t=0$ is feasible ($h\geq 0$), while $f(\hat{x}[i])=\cF_\beta(\hat{x})$. Global optimality of $\hat{x}$ therefore implies $f(0)\geq f(\hat{x}[i])$, hence we have
    \begin{align*}
        \norm{\hat{x}}_D^2+\beta\norm{\hat{x}}_0
        &\leq \norm{\hat{x}-\hat{x}[i]e_i+\mathrm{sign}(\ccr[j])\abs{\hat{x}[i]}\delta_{ij}e_j}_D^2\\
&\quad+\beta\norm{\hat{x}-\hat{x}[i]e_i+ \mathrm{sign}(\ccr[j])\abs{\hat{x}[i]}\delta_{ij}e_j}_0.
    \end{align*}
    The vector $\hat{x}-\hat{x}[i]e_i+\mathrm{sign}(\ccr[j])\abs{\hat{x}[i]}\delta_{ij}e_j$ has at most $\abs{\hat{\sigma}}-1$ nonzero elements since the perturbation removes index $i$ without adding a new nonzero index.  Hence its $\ell_0$-term is at most $\|\hat{x}\|_0-1$, yielding
\begin{align}\label{eq:betabound}
   \notag \beta &\leq \norm{\hat{x}-\hat{x}[i]e_i+\mathrm{sign}(\ccr[j])\abs{\hat{x}[i]}\delta_{ij}e_j}_D^2-\norm{\hat{x}}_D^2\\&\notag=\tpose{(2\hat{x}-\hat{x}[i]e_i+\mathrm{sign}(\ccr[j])\abs{\hat{x}[i]}\delta_{ij}e_j)}D(-\hat{x}[i]e_i+\mathrm{sign}(\ccr[j])\abs{\hat{x}[i]}\delta_{ij}e_j)\\&
    \leq \norm{2\hat{x}-\hat{x}[i]e_i+\mathrm{sign}(\ccr[j])\abs{\hat{x}[i]}\delta_{ij}e_j}_D\norm{-\hat{x}[i]e_i+\mathrm{sign}(\ccr[j])\abs{\hat{x}[i]}\delta_{ij}e_j}_D\\
    &\notag \leq\left(2\norm{\hat{x}}_D+\abs{\hat{x}[i]}\norm{e_i-\mathrm{sign}(\hat{x}[i]\ccr[j])\delta_{ij}e_j}_D\right)\abs{\hat{x}[i]}\norm{e_i-\mathrm{sign}(\hat{x}[i]\ccr[j])\delta_{ij}e_j}_D.
\end{align}
Here, the last two inequalities follow from the Cauchy-Schwarz and the triangle inequalities. By the definition of $\rho_i$, we have $ \|e_i-\mathrm{sign}(\hat{x}[i]\ccr[j])\delta_{ij}e_j\|_D\leq \rho_i$. Substituting this together with $\|\hat x\|_D\leq \sqrt{\eta}$, established in the proof of Proposition~\ref{prop:globup}, into~\eqref{eq:betabound}, we obtain
 \begin{align}
 \label{eq:betaprop}
     \beta\leq (2\sqrt{\eta}+\abs{\hat{x}[i]}\rho_i)\abs{\hat{x}[i]}\rho_i.
 \end{align}
 Consider the polynomial $P(v)= \rho_i^2v^2+2\sqrt{\eta}\rho_i v-\beta$. Since $P$ is a quadratic with a positive leading coefficient, \eqref{eq:betaprop} implies that $|\hat x[i]|$ must be larger than the positive root of $P$. Evaluating these roots gives
 \begin{align*}
     v_+,v_-=\frac{-2\sqrt{\eta}\rho_i\pm\sqrt{4\eta\rho_i^2+4\rho_i^2\beta}}{2\rho_i^2}=\frac{-\sqrt{\eta}\pm\sqrt{\eta+\beta}}{\rho_i}.
 \end{align*}
 Hence, we obtain the lower bound $\abs{\hat{x}[i]}\geq \frac{\sqrt{\eta+\beta}-\sqrt{\eta}}{\rho_i}$.

 \textbf{Case 2: } Assume $\abs{\hat{\sigma}}=1$, then we have a closed form solution for the specific entry $i\in\hat{\sigma}$, which takes the form
\begin{align*}
    \hat{x}[i]=\frac{\bar{\ccr}\;\mathrm{sign}(\ccr[i])}{\abs{\ccr[i]}-\gamma\sqrt{d_i[i]}}\quad\Rightarrow\quad \abs{\hat{x}[i]}=\frac{\bar{\ccr}}{\abs{\ccr[i]}-\gamma\sqrt{d_i[i]}}.
\end{align*}
In both cases the lower bound takes the form
\[
    \min\left\{ \frac{\sqrt{\eta+\beta}-\sqrt{\eta}}{\rho_i},\frac{\bar{\ccr}}{\abs{\ccr[i]}-\gamma\sqrt{d_i[i]}}\right\},
\] where the first term is active when $\abs{\hat{\sigma}}\geq2$ and the second when $\abs{\hat{\sigma}}=1$. Since a global minimizer must fall into one of these two cases, the bound holds unconditionally for every $i\in\hat{\sigma}$.
 \end{proof}
The first term of the bound grows like $\sqrt{\beta}$ for large $\beta$, so the larger the sparsity penalty, the further the nonzero entries of a global minimizer with at least two assets must lie from zero. Having established bounds on the components of global minimizers, we next address their existence, for which we first verify that the objective is coercive.
\begin{lemma}\label{lem:coer}
    The feasible set $\cX$ defined in~\eqref{eq:cX}  is nonempty and closed, and its objective $\cF_\beta(x)$ is coercive on  $\cX$.
\end{lemma}
\begin{proof}
    Since $D \succ 0$, we have
        $x^\top D x \ge s_N(D)\|x\|_2^2$. Hence, $\cF_\beta(x)
        =
        x^\top D x + \beta \|x\|_0
        \ge
        s_N(D)\|x\|_2^2$.
    Therefore, $\cF_\beta$ is coercive. Moreover, the feasible set $\cX
        =
        \{
        x \in \mathbb{R}^N :
        \ccr^\top x - \gamma \|x\|_D \ge \bar \ccr
        \}$ is closed, since it is the preimage of a closed interval under a continuous function. Assumption~\ref{assump:feas} guarantees that $\cX$ is nonempty.

       \end{proof}
We are now ready to establish the existence of a global minimizer.
\begin{theorem}
    Let $\beta>0$. Then the set of global minimizers of $\cF_\beta$,
    \begin{align*}
        \hat{X}
        =
        \left\{
        \hat{x}\in\cX :
        \cF_\beta(\hat{x})
        =
        \min_{x\in\cX}\cF_\beta(x)
        \right\}
    \end{align*}
    is nonempty.
\end{theorem}
\begin{proof}
        By Lemma~\ref{lem:coer}, the feasible set $\cX$ is closed and nonempty, while $\cF_\beta$ is coercive. In addition, $\cF_\beta$ is lower semicontinuous, since $\norm{\cdot}_0$ is lower semicontinuous~\cite[Proof of Proposition 4.3]{Nikolova2013}. Since $\cF_\beta$ is lower semicontinuous and coercive on the closed nonempty set $\cX$, it follows from \cite[Theorem~1.9]{rockafellar1998} that $\cF_\beta$ attains its minimum over $\cX$. Therefore, $\hat{X}$ is nonempty.
\end{proof}
The following statement establishes the existence of a threshold value for the penalty parameter corresponding to each prescribed sparsity level. In particular, for every global minimizer of $\cF_\beta$ associated with a given sparsity, one can identify a penalty parameter that enforces that level.
\begin{proposition} \label{prop:betak}
	For any $1\leq k\leq N-1$, there exists $\beta_k>0$ such that if $\beta > \beta_k$, then every global minimizer $\hat{x}$ of $\cF_\beta$ satisfies $\norm{\hat{x}}_0 \leq k$.
\end{proposition}
\begin{proof}
	Fix $k\in \mathbb I_{N-1}$ and consider the set $X_{k+1} =  \{ x\in \cX: \norm{x}_0 \geq k+1 \}$.	Suppose first that $X_{k+1}$ is nonempty. Then, for each $\overline{x} \in X_{k+1}$, we have $\cF_\beta(\overline{x})
	= \tpose{\overline{x}}D\overline{x}+\beta \norm{\overline{x}}_0
	\geq \beta(k+1)$. Choose any support $\omega \subset \mathbb{I}_N$ such that $|\omega| \leq k$ and $\Xi(\omega) \in \mathcal{X}$. Such a support always exists: by Assumption~\ref{assump:feas} and Remark~\ref{rem:feasall}, every singleton $\{i\}$ with $i \in \II_N$ satisfies
$H_{\{i\}} = \frac{|\mathfrak{r}[i]|}{\sqrt{d_i[i]}} > \gamma$,
which guarantees $\Xi(\{i\}) \in \mathcal{X}$. Since $|\{i\}| = 1 \leq k$ for any $k \leq N-1$, taking $\omega = \{i\}$ for any $i \in \mathbb I_N$ yields a valid choice. Choose $\beta_k$ so that
	$\beta_k \geq \tpose{\Xi(\omega)}D\Xi(\omega)$. For such a choice,
	\begin{align*}
		\cF_\beta (\Xi(\omega))
		&= \tpose{\Xi(\omega)}D\Xi(\omega) + \beta \norm{\Xi(\omega)}_0
		\leq \beta_k + \beta k
		< \beta(k+1)
		\leq \cF_\beta(\overline{x}),
		\qquad \forall \overline{x}\in X_{k+1}
	\end{align*}
	whenever $\beta>\beta_k$. Let $\hat{x}\in \hat{X}$ denote a global minimizer of $\cF_\beta$ with
	$\hat{x}\in \cX$. Since
	\begin{align*}
	\cF_\beta(\hat{x})
	\leq \cF_\beta(\Xi(\omega))
	< \cF_\beta(\overline{x}),
	\qquad \forall \overline{x}\in X_{k+1},
	\end{align*}
	we must have $\hat{x}\notin X_{k+1}$. By the definition of $X_{k+1}$, this implies $\norm{\hat{x}}_0 \leq k$. If $X_{k+1} = \emptyset$, then the existence of a global minimizer immediately yields $\norm{\hat{x}}_0\leq k$.
\end{proof}
\begin{remark}\label{rem:noofsols}
    For any $\beta>0$, every admissible support induces a local minimizer of $\cF_\beta$. Nevertheless, this correspondence is not one-to-one, as a single local minimizer may be generated by multiple supports via zero-padding operators. Consequently, the number of distinct local minimizers is bounded above by $2^N - 1$, which corresponds to the total number of nontrivial supports.
\end{remark}
In the next section, we extend the theoretical results developed here to the problem of robust return maximization under a variance budget. In particular, we adapt the main structural and sparsity-related properties to this risk-constrained setting and show that similar conclusions can be obtained in that framework.
\section{Analysis of \texorpdfstring{\eqref{prob:cprmvp2}}{P2}}\label{theory2}
In this section, we characterize the support-restricted, local, and global minimizers of~\eqref{prob:cprmvp2} and derive bounds on the nonzero components of its globally optimal solutions. For ease of notation, define its feasible set as
\begin{align}
\cY:=\left\{x\in\RR^N:\norm{x}_D \leq T\right\},\label{eq:cY}
\end{align}
We next impose a lower bound on $T$ that rules out the zero vector as a global
minimizer.
\begin{assumption}\label{assum:feas2}
    Under Assumption~\ref{assump:feas}, we further assume
    \begin{align*}
        T>\min_{i\in\II_N}\left\{\frac{\sqrt{d_i[i]}\beta}{\abs{\ccr[i]}-\sqrt{d_i[i]}\gamma}\right\}.
    \end{align*}
\end{assumption}
\begin{remark}
    The purpose of this lower bound is to rule out the trivial solution $x=\mathbf{0}$, in which all wealth is invested in the risk-free asset. To see this, fix $i\in\II_N$ and consider the singleton portfolio
    \[
    x=\frac{T\,\mathrm{sign}(\ccr[i])}{\sqrt{d_i[i]}}e_i.
    \]
    This portfolio satisfies $\norm{x}_D=T$, and its objective value is
    \[
    \cG_\beta(x)
    =
    -\frac{T}{\sqrt{d_i[i]}}
    \left(\abs{\ccr[i]}-\gamma\sqrt{d_i[i]}\right)+\beta <0,
    \]
    where the strict inequality follows from Assumption~\ref{assum:feas2}. Thus, $\cG_\beta$ takes a negative value at a feasible point in $\cY$, whereas $\cG_\beta(\mathbf{0})=0$. Hence, the zero vector cannot be globally optimal.
\end{remark}

\begin{remark}The reasoning outlined in Remark~\ref{rem:whyworks} remains applicable to this objective function, as it preserves separability with respect to supports. This structural property plays a central role in developing a rigorous characterization of local minimizers.
\end{remark}
We consider the following restricted problem in order to characterize the local minimizers of~\eqref{prob:cprmvp2}:
\begin{align}\tag{$\cP^2_\omega$}\label{prob:mvp2-omega}
\min_{x\in \cY\cap K_\omega}\gamma\norm{x}_D-\tpose{\ccr}x.
\end{align}

Under Assumption~\ref{assump:feas}, we have $H_\omega>\gamma$ for every nonempty $\omega$. In this case, the subproblem admits a unique optimal solution, which we characterize in the next section.
\subsection{Minimizers of \texorpdfstring{\eqref{prob:mvp2-omega}}{P2w}}
For a fixed nonempty index set $\omega\subseteq\II_N$, we introduce the corresponding restricted feasible region
$\cY_\omega=\{u\in\RR^{\abs{\omega}}:\norm{u}_{D_\omega}\leq T\}$. Based on this construction and the zero padding operator, we formulate a problem equivalent to \eqref{prob:mvp2-omega}, now expressed over a convex feasible set:
\begin{align}\tag{$\cZ\cP^2_\omega$}\label{prob:zp-rmvp2}
\min_{u\in \cY_\omega}\gamma\norm{u}_{D_\omega}-\tpose{\ccr_{\omega}}u, \quad \omega\neq \emptyset.
\end{align}
In this case, the subproblems~\eqref{prob:mvp2-omega} and~\eqref{prob:zp-rmvp2} admit a unique optimal solution, characterized below.
\begin{proposition}\label{prop:xi2}
For nonempty $\omega \subseteq \II_N$, the unique optimal solution of \eqref{prob:zp-rmvp2} is
\begin{align*}
\pi(\omega):=\frac{T}{H_\omega}(D_\omega)^{-1}\ccr_\omega.
\end{align*}
\end{proposition}
\begin{proof}
    A proof of this result is given in \cite[Proposition~2]{Pinar2016}.
\end{proof}
\begin{remark}
    For $\omega\subseteq\mathbb{I}_N$, we write $\Pi(\omega)=Z_\omega(\pi(\omega))$ for its zero-padding to $\mathbb{R}^N$, similarly as in Remark~\ref{rem:zpdef}.
\end{remark}
If $H_\omega<\gamma$, then $u=0$ is the unique optimal solution of~\eqref{prob:zp-rmvp2}, that is, all wealth is held in the risk-free asset. Assumption~\ref{assump:feas} excludes this degenerate case, since it guarantees $\min_{\omega\neq\emptyset} H_\omega > \gamma$ for every nonempty $\omega$.
As in the risk-minimization case (\S\ref{theory1}), the next proposition gives the optimal multiplier of the subproblem constraint and shows that the constraint is active at optimality. The multiplier is used in the branching rule of the algorithm.
\begin{proposition}\label{prop:dual_solution_subproblem}
For nonempty $\omega\subseteq \II_N$ the optimal Lagrange multiplier associated with the constraint of~\eqref{prob:zp-rmvp2} is $\mu^* = H_\omega- \gamma$. Moreover, the constraint is active at the optimal solution $\hat u$, \textit{i.e.}, $\norm{\hat{u}}_{D_\omega}=T$.
\end{proposition}
\begin{proof}
The argument follows closely that of \cite[Proposition~2]{Pinar2016}. In addition, we establish that the restricted constraint $\norm{u}_{D_\omega}\leq T$ of~\eqref{prob:zp-rmvp2} is active at optimality. Suppose that the constraint is inactive at $\hat u$. Complementary slackness then gives $\mu^*=0$, so that $\hat u$ minimizes the unconstrained convex function $u\mapsto\gamma\norm{u}_{D_\omega}-\tpose{\ccr_\omega}u$. If $\hat u\neq 0$, stationarity gives $\gamma D_\omega\hat u/\norm{\hat u}_{D_\omega}=\ccr_\omega$ and taking the $D_\omega^{-1}$-norm of both sides yields $H_\omega=\norm{\ccr_\omega}_{(D_\omega)^{-1}}=\gamma$. If $\hat u=0$, the optimality condition $\ccr_\omega\in\gamma\,\partial\norm{\cdot}_{D_\omega}(0)$ gives $H_\omega\leq\gamma$. Both conclusions contradict $H_\omega>\gamma$, which is guaranteed by Assumption~\ref{assump:feas}. Consequently, the optimal solution $\hat{u}$ of \eqref{prob:zp-rmvp2} satisfies $\norm{\hat{u}}_{D_\omega}=T$.

Stationarity at the nonzero optimal solution now gives $\ccr_\omega
=(\gamma+\mu^*) \frac{D_\omega\hat u}{\norm{\hat u}_{D_\omega}}.$
Taking the $D_\omega^{-1}$-norm and using $\norm{\hat u}_{D_\omega}=T$ yields
$H_\omega=\gamma+\mu^*$. Hence, $\mu^*=H_\omega-\gamma$.
 \end{proof}
\begin{lemma}\label{lemma:equiv2}
Problems \eqref{prob:zp-rmvp2} and \eqref{prob:mvp2-omega} are equivalent.
\end{lemma}
\begin{proof}
The argument proceeds along the same lines as in Lemma~\ref{lemma5}.
\end{proof}
\subsection{(Local) Minimizers of \texorpdfstring{\eqref{prob:cprmvp2}}{P2}}
To analyze the local minimizers of~\eqref{prob:cprmvp2}, we express its
objective using the indicator function $\phi$; \textit{i.e.}, $\cG_\beta(x)=\gamma\norm{x}_D-\tpose{\ccr}x+\beta\sum_{i\in\sigma(x)}\phi(x[i])$. The following proposition, analogous to Proposition~\ref{choiceofrho}, identifies a neighborhood in which activating components outside the support
of a feasible point cannot decrease the objective of~\eqref{prob:cprmvp2}.
\begin{proposition}\label{choiceofrho2}
Let $\beta>0$ and $\hat{x}\in\cY\setminus\{0\}$. Define $\hat{\sigma}=\sigma(\hat{x})$ and
\begin{align*}
\rho:=\min\left\{\min_{i\in\hat{\sigma}}|\hat{x}[i]|,\ \frac{\beta}{\norm{\ccr}_1+\gamma\sqrt{N\norm{D}_2}+1}\right\}.
\end{align*}
Then $\rho>0$, and:
\begin{enumerate}[label=(\roman*)]
\item If $y\in B_\infty(0,\rho)$, then
$\sum_{i\in\II_N}\phi(\hat{x}[i]+y[i])=\sum_{i\in\hat{\sigma}}\phi(\hat{x}[i])+\sum_{i\in\hat{\sigma}^c}\phi(y[i])$.
\item If $y\in B_\infty(0,\rho)\cap(\RR^N\setminus K_{\hat{\sigma}})$, then $
\cG_\beta(\hat{x}+y)\geq \cG_\beta(\hat{x})$, with strict inequality whenever $\hat{\sigma}^c\neq\emptyset$.
\end{enumerate}
\end{proposition}
\begin{proof}
    We first prove $(i)$.  For $y\in B_\infty(0,\rho)$ we have $\norm{y}_\infty<\min_{i\in\hat{\sigma}}\abs{\hat{x}[i]}$. This implies $ \phi(\hat{x}[i]+y[i])=\phi(\hat{x}[i])$ for $i\in \hat{\sigma}$. For $i\in \hat{\sigma}^c$ we have $\phi(\hat{x}[i]+y[i])=\phi(y[i])$, which  gives the desired result.

        We now prove $(ii)$. Let $ y\in B_\infty(0,\rho)\setminus K_{\hat{\sigma}} $. Then
			\begin{align*}
				\cG_\beta(\hat{x}+y)&=\gamma\norm{\hat{x}+y}_D-\tpose{\ccr}(\hat{x}+y)+\beta\norm{\hat{x}+y}_0\\&
				=\cG_\beta(\hat{x})-\tpose{\ccr}y+\gamma(\norm{\hat{x}+y}_D-\norm{\hat{x}}_D)+\beta\sum_{i\in\hat{\sigma}^c}\phi(y[i])\\&
				\geq \cG_\beta(\hat{x})-\tpose{\ccr}y-\gamma\norm{y}_D+\beta\norm{y_{\hat{\sigma}^c}}_0\\&
				\geq\cG_\beta(\hat{x})-\norm{y}_\infty(\norm{\ccr}_1+\gamma\sqrt{N\norm{D}_2})+\beta\norm{y_{\hat{\sigma}^c}}_0.
			\end{align*}
					For $ \hat{\sigma}^c=\emptyset $, inequality is trivial. If not, $ \norm{y_{\hat{\sigma}^c}}_0\geq 1 $, and the given radius   provides the inequality.
\end{proof}
The next result establishes a correspondence, analogous to Proposition~\ref{prop:localmin} and Lemma~\ref{lemma8}, between local minimizers of~\eqref{prob:cprmvp2} and global minimizers of~\eqref{prob:mvp2-omega} for a given $\omega \subseteq \II_N$, and the proofs are omitted accordingly.
\begin{proposition}\label{prop:localmin2}
Let $\omega\subseteq\II_N$, $\omega\neq\emptyset$. For any $\beta>0$, the objective $\cG_\beta$ reaches a (local) minimum of \eqref{prob:cprmvp2} at $\Pi(\omega)$, with $|\sigma(\Pi(\omega))|\geq 1$ and $\sigma(\Pi(\omega))\subseteq \omega$.
Moreover, any nonzero (local) minimizer $\hat{x}$ of \eqref{prob:cprmvp2} satisfies $\hat{x}=\Pi(\sigma(\hat{x}))$.
\end{proposition}
\subsection{Global Minimizers of \texorpdfstring{\eqref{prob:cprmvp2}}{P2}}
We begin with an observation concerning the existence of an upper bound on the components of globally optimal portfolios.
\begin{remark}
    Let $\beta>0$, and let $\hat{x}$ be a global minimizer of \eqref{prob:cprmvp2}. In Proposition~\ref{prop:globup} we developed a non-trivial upper bound on the nonzero entries of a global minimizer of \eqref{prob:cprmvp}. For~\eqref{prob:cprmvp2}, the feasible region $\cY$ defined in~\eqref{eq:cY}  is compact. Thus, the corresponding bound is immediate. Indeed, every global minimizer $\hat x$ satisfies $\norm{\hat{x}}_D= T$ and $\norm{\hat{x}}_\infty\leq \frac{T}{\sqrt{s_N(D)}}<\infty$.
\end{remark}
We proceed by establishing lower bounds on the nonzero components of a globally optimal portfolio.
\begin{theorem}\label{thm:globup2}
    Let $\beta>0$ and suppose $\hat{x}$ is a global minimizer of~\eqref{prob:cprmvp2}. If $\hat{\sigma}=\sigma(\hat{x})$ denotes the support of $\hat{x}$, then for every $i\in\hat{\sigma}$, we have
   \begin{align*}
    \abs{\hat{x}[i]}\geq \min\left\{\frac{\beta}{\abs{\ccr[i]}+H\sqrt{d_i[i]}},\frac{T}{\sqrt{d_i[i]}}\right\}.
\end{align*}
\end{theorem}
\begin{proof}
    Let $\abs{\hat{\sigma}}\geq 2$ and  $i\in\hat{\sigma}$. We define $g_{i}:\RR^N\rightarrow \RR^N$ as $g_{i}(x)=x-x[i]e_i$ where $e_i$ is the canonical basis vector of the specified index. We introduce the function $f(t_i)=\cG_\beta(g_{i}(\hat{x})+t_ie_i)$. Finally, we define a feasibility function $h(t_i)=T-\norm{g_{i}(\hat{x})+t_ie_i}_D$.
   Since $\hat{x}$ is a global minimizer, it should be a minimizer of $(\cP^2_{\hat \sigma})$. As shown in Proposition~\ref{prop:dual_solution_subproblem}, the restricted constraint is active at $\hat{x}$. This implies that the main constraint is also active, yielding $h(\hat{x}[i])= 0$. Then, we have the following three cases:
   
   \textbf{Case I: } $h(0)\geq 0$, in this case we have $\norm{g_{i}(\hat{x})}_D\leq T$ and due to the global optimality of $\hat{x}$ we have
   \begin{align*}
       \gamma \norm{g_{i}(\hat{x})}_D-\tpose{\ccr}g_{i}(\hat{x})+\beta\norm{g_{i}(\hat{x})}_0\geq  \gamma \norm{\hat{x}}_D-\tpose{\ccr}\hat{x}+\beta\norm{\hat{x}}_0.
   \end{align*}
   These together imply $\abs{\hat{x}[i]}\geq\frac{\beta}{\abs{\ccr[i]}}$.
   
   \textbf{Case II: } $h(0)<0$ and $f(0)\geq f(\hat{x}[i])$. Then, similarly we have $\abs{\hat{x}[i]}\geq \frac{\beta}{\gamma\sqrt{d_i[i]}+\abs{\ccr[i]}}$.
   
   \textbf{Case III: } $f(0)<f(\hat{x}[i])$ and $h(0)<0$. Since $g_i(\hat{x})\neq 0$ we can define $\hat{u}=\frac{T}{\norm{g_i(\hat{x})}_D}g_i(\hat{x})$. $\hat{u}$ is a feasible point so we require $\cG_\beta(\hat{x})\leq \cG_\beta(\hat{u})$. Then we have
\begin{align*}
    \gamma T - \tpose{\ccr}\hat{x} + \beta\norm{\hat{x}}_0 \leq \gamma T - \tpose{\ccr}\hat{u} + \beta(\norm{\hat{x}}_0 - 1) \quad\Rightarrow\quad
    \beta \leq \tpose{\ccr}(\hat{x} - \hat{u}).
\end{align*}
Substituting the decomposition $\hat{x} = g_i(\hat{x}) + \hat{x}[i]e_i$ and the definition of $\hat{u}$, we obtain:
\begin{align*}
    \beta &\leq \tpose{\ccr}\left(g_i(\hat{x}) + \hat{x}[i]e_i - \frac{T}{\norm{g_i(\hat{x})}_D} g_i(\hat{x})\right) = \ccr[i]\hat{x}[i] + \left(1 - \frac{T}{\norm{g_i(\hat{x})}_D}\right)\tpose{\ccr}g_i(\hat{x}) \\
    &= \ccr[i]\hat{x}[i] + \left(\norm{g_i(\hat{x})}_D - T\right) \frac{\tpose{\ccr}g_i(\hat{x})}{\norm{g_i(\hat{x})}_D}.
\end{align*}
Since $h(0) < 0$, we have $\norm{g_i(\hat{x})}_D > T$, so the coefficient $(\norm{g_i(\hat{x})}_D - T)$ is positive. By the definition of the dual norm, we have $\frac{\tpose{\ccr}g_i(\hat{x})}{\norm{g_i(\hat{x})}_D} \leq H$. Applying this upper bound yields:
\begin{align*}
    \beta &\leq \ccr[i]\hat{x}[i] + (\norm{g_i(\hat{x})}_D - T)H.
\end{align*}
Next, using the triangle inequality $\norm{g_i(\hat{x})}_D \leq \norm{\hat{x}}_D + \norm{\hat{x}[i]e_i}_D$ and noting that $\norm{\hat{x}}_D=T$, we have $\norm{g_i(\hat{x})}_D - T \leq \norm{\hat{x}[i]e_i}_D = \abs{\hat{x}[i]}\sqrt{d_i[i]}$.
Substituting this back into the inequality and using $\ccr[i]\hat{x}[i] \leq \abs{\ccr[i]}\abs{\hat{x}[i]}$:
\begin{align*}
    \beta &\leq \abs{\ccr[i]}\abs{\hat{x}[i]} + \abs{\hat{x}[i]}\sqrt{d_i[i]}H = \abs{\hat{x}[i]} \left( \abs{\ccr[i]} + \sqrt{d_i[i]}H \right) \;\;\Rightarrow\;\;\abs{\hat{x}[i]} \geq \frac{\beta}{H\sqrt{d_i[i]} + \abs{\ccr[i]}}.
\end{align*}
We observe that this lower bound is suitable for all three cases above.
Finally, if $\abs{\hat{\sigma}}=1$ then we have closed form solutions for the specific entry $i\in\hat{\sigma}$
\begin{align*}
    \hat{x}[i]=\frac{T\mathrm{sign}(\ccr[i])}{\sqrt{d_i[i]}}\quad\Rightarrow\quad \abs{\hat{x}[i]}=\frac{T}{\sqrt{d_i[i]}}.
\end{align*}
Combining the bounds from the $|\hat{\sigma}|\geq 2$ and $|\hat{\sigma}|=1$ cases yields the desired result.
   \end{proof}
   \begin{remark}
       The lower bound in Theorem~\ref{thm:globup2} is simpler than the one in Theorem~\ref{thm:lower-bound}: it is linear in $\beta$, does not involve the pairwise quantities $\rho_i$, and depends on the data only through $\ccr[i]$, $d_i[i]$, $H$ and $T$. In our experiments it also separated the nonzero components of global minimizers from zero more clearly, which is the property exploited by the warm-start heuristic of Section~\ref{sec:BnBwarmstart}.
   \end{remark}
Because the feasible set is compact, the existence of a global minimizer can be verified more straightforwardly than in the earlier case. We first establish that the current objective function is lower semi-continuous for any selection of problem parameters.
   \begin{lemma}
       For $\beta>0$, the robust return maximization objective $\cG_\beta$ is lower semi-continuous.
   \end{lemma}
   \begin{proof}
      The function $\norm{\cdot}_0 : \RR^N \rightarrow \RR$ is lower semi-continuous~\cite[Proof of Proposition~4.3]{Nikolova2013}. Consequently, the objective function $\cG_\beta(x) = \gamma \norm{x}_D - \tpose{\ccr} x + \beta \norm{x}_0$ is also lower semi-continuous, as it is a sum of lower semi-continuous functions.
   \end{proof}
   \begin{theorem}
	Let $\beta>0$. Then the set of global minimizers of $\cG_\beta$,
	\begin{align*}
		\hat{Y}=\left\{\hat{x}\in\cY\;:\;\cG_\beta(\hat{x})=\min_{x\in\cY}\cG_{\beta}(x)\right\}
	\end{align*}
	is nonempty.
\end{theorem}
\begin{proof}
    This result follows directly from the extended form of the Weierstrass' extreme value theorem: a lower semicontinuous function attains its minimum over a compact feasible set.
\end{proof}
Finally, before concluding the theoretical developments, we observe that Proposition~\ref{prop:betak} and Remark~\ref{rem:noofsols} extend directly to the robust return maximization problem. Consequently, for any prescribed sparsity level $k$, there exists a regularization parameter $\beta_k > 0$ that ensures the desired sparsity level in globally optimal portfolios. Furthermore, globally optimal portfolios can be identified among $2^N - 1$ locally optimal candidates, where this count arises from considering only nontrivial support sets.

In the next section, we provide a brief description of the proposed algorithm and the associated heuristic procedures.
\section{Enumeration Based BnB Algorithm}\label{BBalgo}
In this section, we develop an enumeration-based branch-and-bound algorithm for~\eqref{prob:cprmvp} and~\eqref{prob:cprmvp2}, following the scheme of \cite{SenAkkayaPinar2025} for the nonrobust counterpart and adapting it to the robust setting through the bounds established in Sections~\ref{theory1} and~\ref{theory2}. We further refine the scheme with an additional pruning step in the bounding of right nodes, which allows an entire subtree to be replaced by a single leaf evaluation.

We denote by $P \subseteq\II_N$ the set of candidate assets that remain, whether or not the heuristic is applied, and decompose the problem into subproblems~\eqref{prob:mvp-omega} over support subsets $\omega \subseteq P$. Each subproblem characterizes the local minimizers supported on $\omega$ and corresponds to a node in the enumeration tree. At each node we solve the lower-dimensional equivalent subproblem \eqref{prob:zpmvp} and maintain a five-tuple $(x, lb, ub, P, S)$, whose components are as follows.
\begin{itemize}
    \item $x$ represents the solution obtained by solving \eqref{prob:zpmvp} for a nonempty support subset $\omega$. If $\omega=\emptyset$, then the node is pruned.
    \item $lb$ represents a \emph{lower score}, which is less than or equal to the upper bound at each node. For ease of reference, it will be referred to as the lower bound for the remainder of the paper.
    \item $ub$ represents an upper bound on the optimal value obtained from the corresponding node.
   \item $P$ denotes the set of candidate assets whose inclusion is still undecided, from which the branching asset is selected. 
    \item $S$ denotes the set of assets already fixed in the support at this node.
\end{itemize}
 Each node produced by the algorithm is inserted into a \emph{priority queue}, with its priority determined by the previously defined lower bounds. At the beginning of the next iteration, the node with the highest priority
(\textit{i.e.}, the lowest lower bound) is removed from the priority queue, and the
values of $x$, $lb$, $ub$, $P$, and $S$ are updated according to this node.
If multiple nodes share the same lowest lower bound, the node that was added
to the queue first is selected. The branching process then proceeds from this
chosen node, allowing the algorithm to systematically explore the search space. In the following subsections, we outline each step of the algorithm and present Algorithm~\ref{algo:pseudo1}, which summarizes the entire procedure.
All subroutines can be applied to \eqref{prob:cprmvp2} with minor modifications,
and in the following sections we present computational results for both problems.
\subsection{Warm-Start Heuristic}\label{sec:BnBwarmstart}
Computation time increases with larger $N$, because the algorithm may encounter many suboptimal solutions, and the number of such solutions affects efficiency. To address this, a warm-start heuristic is introduced to reduce memory usage and problem size.

The method first solves the problem with the full support set and determines a conservative elimination level using componentwise lower bounds. Then, for each asset $i\in P$, we compute the deficit $b[i]-|u_P[i]|$ between the componentwise lower bound $b[i]$ of Theorem~\ref{thm:lower-bound} and the weight $u_P[i]$ that asset $i$ receives in the solution $u_P=\xi(P)$ on the full candidate set, and remove the assets with the largest deficits from the candidate support set, thereby reducing the dimension of the problem early in the process. For instance, if the goal is a 10-sparse solution among 60 variables, about 30\% of variables might be eliminated rather than removing all but 10, in order to avoid excluding potentially optimal components.

This approach speeds up computation but introduces a trade-off between solution quality and runtime, so the elimination level must be chosen carefully. Guidance for this choice can come from sparsity information provided by Proposition~\ref{prop:betak}. If no warm-start is applied, that is, if no support is eliminated,  the method reduces to the full branch-and-bound algorithm. The validity of the associated bounding and pruning scheme follows from the generic enumeration argument in~\cite[Appendix~B]{SenAkkayaPinar2025}; consequently, when the algorithm terminates, the returned incumbent is globally optimal up to the prescribed stopping tolerance. The heuristic is most beneficial when the true solution is sparse, which aligns with portfolio optimization practice, since sparse portfolios reduce transaction costs.
\subsection{Branching}
When a node is taken from the queue for examination, the first step is to determine the most promising asset by analyzing the gradient of the Lagrangian function of \eqref{prob:zpmvp}. This branching strategy evaluates the quality of the current solution and aims to improve it. If the selected asset set at a node is empty, the procedure is modified by choosing the index that minimizes the variance-to-return ratio.

For branching, we represent the Lagrangian function $L$ by associating a Lagrange multiplier with the constraint defining $\cX_\omega$ and compute its gradient $(\nabla L)$ using Proposition~\ref{prop:dual_solution}. We then identify the most promising candidate among the set of possible assets. If the set of selected assets is nonempty, \textit{i.e.}, $S\neq\emptyset$, we select
$
j \in \arg\max_{i\in P} \abs{\nabla L_i}.
$
In this case, rather than using the covariance matrix $D_P$, we extract the submatrix $D_{P,S}$, whose rows correspond to indices in $P$ and columns correspond to indices in $S$, ensuring dimensional compatibility in the gradient computation.

If $S=\emptyset$, the gradient rule is unavailable, since the submatrix $D_{P,S}$ is empty. In this case, we apply an alternative branching rule based on the variance-to-return ratio of assets in $P$, and select
$
j \in \arg\min_{i\in P} \mathrm{diag}(D_P)[i] / \ccr_{P}[i].
$
This rule favors assets with relatively low variance and high return.

In both cases, an asset $j$ is selected from the candidate set $P$, and the algorithm generates two distinct nodes: a \emph{left} node in which the chosen asset is included in the support
$
S^L \leftarrow S \cup \{j\},
$
and a \emph{right} node in which the asset is excluded from the support
$
P^R \leftarrow P \setminus \{j\}$, $
S^R \leftarrow S \cup P^R
$.
As a result, at every iteration the algorithm identifies the asset that appears most influential for the portfolio, which helps accelerate convergence relative to a standard BnB procedure. The primary objective is to obtain solutions with as few nonzero components as possible, that is, with small support sets.
\subsection{Bounding}
At the start of each iteration, the algorithm removes from the queue the tuple with the highest priority and updates the corresponding values $(x, lb, ub, P, S)$.

For left nodes, an upper bound is obtained by solving the subproblem with fixed support $S^L$ and evaluating an upper estimate using the objective value of \eqref{prob:cprmvp} at the resulting solution, $ub^L=\tpose{\xi(S^L)}D_{S^L}\xi(S^L)+\beta|{S^L}|$. The lower bound of a left node is computed by adding the sparsity penalty $\beta$ to the parent node’s lower bound, $lb^L=lb+\beta$. Depending on the triviality of the support, the node is then added to the queue with the updated bounds.

For right nodes, the upper bound is inherited from the parent node because the branching step excludes an asset from the candidate set without generating a new feasible solution with fixed support. The lower bound is obtained from the relaxed subproblem of type \eqref{prob:zpmvp}. At initialization, a global lower bound is obtained by solving the problem with $\omega=\II_N$, since this formulation omits the sparsity penalty and considers all indices. Following \cite{SenAkkayaPinar2025}, after excluding the branching asset, the right child is assigned the bounds
\begin{align}\label{eq:right_lb}
     ub^R=ub,
    \qquad
    lb^R = \xi(S^R)^\top D_{S^R}\xi(S^R) + \beta \abs{S}.
\end{align}
In the scheme of \cite{SenAkkayaPinar2025}, the right child is then enqueued. Here, before enqueuing the right child, we apply the following additional pruning step. We first check whether
\[
    lb^R+\beta \geq ub^*-\epsilon,
\]
where $ub^*$ denote the objective value of the incumbent, that is, the best feasible solution found so far, and $\epsilon$ is the prescribed optimality tolerance. 
If this condition holds, then the nonterminal descendants of the right subtree are pruned. The following proposition justifies this additional pruning step.
\begin{proposition}\label{prop:pruning}
    Consider \eqref{prob:cprmvp} and let a node be represented by the current support set $S$ and candidate set $P$. Let $j\in P$, $P^R=P\setminus\{j\}$, and $S^R=S\cup P^R$. Then $lb^R$ as defined in~\eqref{eq:right_lb} satisfies the following property: if $lb^R+\beta \geq ub^*-\epsilon$, then every feasible portfolio $x\in\cX$ such that $S\subsetneq \sigma(x)\subseteq S^R$ satisfies $\cF_\beta(x)=\tpose{x}Dx+\beta\|x\|_0 \geq ub^*-\epsilon$. Consequently, every descendant support $\tilde{S}$ with $S\subsetneq \tilde{S}\subseteq S^R$ may be pruned.
\end{proposition}
\begin{proof}
    Let $x\in\cX$ satisfy $S\subsetneq \sigma(x)\subseteq S^R$, and define $\tilde{S}=\sigma(x)$. Since $\tilde{S}\subseteq S^R$, all nonzero components of $x$ lie in $S^R$. Therefore, the restriction of $x$ to the indices in $S^R$ is feasible for~\eqref{prob:zpmvp} with $\omega=S^R$. By optimality of $\xi(S^R)$, we obtain
    \[
        \tpose{x}Dx \geq \xi(S^R)^\top D_{S^R}\xi(S^R).
    \]
    Since $S\subsetneq \tilde{S}$, we also have $\|x\|_0 = |\tilde{S}| \geq |S|+1$. Hence, combining the two inequalities, we get
    \[
      \cF_\beta(x)=\tpose{x}Dx+\beta\|x\|_0 \geq \xi(S^R)^\top D_{S^R}\xi(S^R)+\beta(|S|+1)=lb^R+\beta.
    \]
    Thus, if $lb^R+\beta\geq ub^*-\epsilon$, then $\cF_\beta(x) \geq ub^*-\epsilon$ for every feasible $x\in\cX$ with $S\subsetneq \sigma(x)\subseteq S^R$. Therefore, no feasible portfolio associated with a descendant support $\tilde{S}$ satisfying $S\subsetneq \tilde{S}\subseteq S^R$ can improve the incumbent by more than $\epsilon$. Consequently, every such descendant support $\tilde{S}$ may be pruned.
\end{proof}
The effect of this pruning step can also be quantified. Indeed, consider the right child defined by branching on $j$, namely the node with candidate set $P^R=P\setminus\{j\}$ and current support $S$. Under the standard enumeration scheme~\cite{SenAkkayaPinar2025}, this node would be enqueued and explored. Since its support remains $S$ and only indices from $P^R$ remain available for future branching, every support generated in that subtree satisfies $S\subseteq \tilde S \subseteq S^R$. When $lb^R+\beta\geq ub^*-\epsilon$, Proposition~\ref{prop:pruning} shows that all supports with $S\subsetneq \tilde S \subseteq S^R$ may be pruned. The number of such supports is
\[
    \sum_{i=1}^{|P^R|}\binom{|P^R|}{i}=2^{|P^R|}-1.
\]
Therefore, the additional pruning step may remove an \emph{exponentially} large portion of the right subtree in a single iteration. An analogous statement holds for \eqref{prob:cprmvp2} after replacing $\tpose{x}Dx$ by $\gamma\|x\|_D-\tpose{\ccr}x$ and $\xi$ by $\pi$. The details are omitted for brevity.
\subsection{Termination}
The algorithm terminates when there are no remaining nodes to explore or when the best upper bound $ub^*$ found during the search matches the lower bound $lb$ of the current node within a user-defined tolerance. In this case, further exploration cannot yield a better solution, and the procedure stops early.
\RestyleAlgo{ruled}
\begin{algorithm}[!htbp]
\caption{A Branch-and-Bound Algorithm for \eqref{prob:cprmvp}}
\label{algo:pseudo1}
\parbox{\textwidth}{
\begin{minipage}[t]{0.48\textwidth}
\,\\
\textbf{Input:} tolerance $\epsilon=1E-8$, local and global upper bounds $ub=ub^* = \infty $, candidate support set $P = \left\{1,\ldots,N\right\}$, current support set $S = \emptyset $, fixed-support solution $u_{P}=\xi(P)$, lower bound $lb$ from solving \eqref{prob:zpmvp} with $\omega=P$, the corresponding componentwise lower bound $b$ from Theorem~\ref{thm:lower-bound}, and queue $q=\emptyset$. \\
\textbf{Warm-Start:} let $J_k\subseteq P$ be the set of indices corresponding to the $k$ largest values of $b[i]-|u_P[i]|$, and update $P \leftarrow P \setminus J_k$.
Enqueue $(u_P,lb,ub,P,S)$ into $q$. \\
\While{$q\neq \emptyset$ \textbf{and} $ub^* - lb > \epsilon $}{
Extract the highest-priority tuple $(x,lb,ub,P,S)$ from $q$.\\
\If{$ub<ub^*$}{Set $ub^* \gets ub$ and $x^* \gets x$.}
$q\gets\textsc{Branch\&Bound}(q,x,lb,ub,P,S)$.
}
\return{$ub^*$ and its corresponding solution $x^*$.}
\end{minipage}
\hfill
\begin{minipage}[t]{0.5\textwidth}
\,\\
\textbf{Subroutine:} $\textsc{Branch\&Bound}(q,x,lb,ub,P,S)$\\
\textbf{Input:} queue $q$ and node $(x,lb,ub,P,S)$.\looseness=-1\\
\If{$P\neq \emptyset$}{
\Branching{\\
\eIf{$\abs{S}\geq 1$}{Select $j\in\mathrm{argmax}_{i\in P} |\nabla L_i|$.}
{Select  $j\in\arg\min_{i\in P}\frac{\mathrm{diag}(D_P)[i]}{\ccr_P[i]} $.}
Left node: $S^L \gets S \cup \{j\}$.\\
Right node: $P^R \gets P\setminus \!\{j\}$, $S^R\gets P^R \cup S$.
}
\BoundingR{\\
\If{$\abs{S^R}\geq 1$}{$lb^R=\xi(S^R)^\top D_{S^R} \xi(S^R) + \beta \abs{S}$. \\
\If{$lb^R+\beta<ub^*-\epsilon$}{Enqueue $(x,lb^R,ub, P^R, S)$ into $q$.}
}
}
\BoundingL{\\
\If{$\abs{S^L}\geq 1$}{$ub^L = \tpose{\xi(S^L)} D_{S^L} \xi(S^L)+\beta \abs{S^L}$. \\
\!Enqueue ($\xi(\!S^L),lb +\beta,ub^L\!, P^R\!, S^L$) into $q$.}
}
}
\return{updated queue $q$.}
\end{minipage}
}
\end{algorithm}
\section{Computational Results}\label{computation}
The algorithms were evaluated using datasets sourced from \cite{CesaroneEtal2025} and \cite{AraratEtal2024}. The former includes daily price data (adjusted for dividends and stock splits) for the DowJones, EuroStoxx50, FTSE100, NASDAQ100, and S\&P 500 indices. The latter comprises ETF, Eurobonds, and Italian Bonds datasets, providing daily asset returns derived from prices and total returns adjusted for dividends and splits. Summary statistics of these datasets are presented in Table~\ref{tab:datasets}.

Experiments were run on a Linux cluster using a single CPU core and no GPU. The runs were carried out on nodes equipped with Intel Xeon Gold 6240 processors and 376 GiB RAM. The implementation used Python and Gurobi 13.0.2. The code used to generate all reported results is available at \url{https://github.com/ecyayla/robust-mean-variance-portfolio-optimization}. We benchmarked our method, enhanced with the warm-start heuristic, against Gurobi using the mixed-integer second-order cone reformulations of~\eqref{prob:cprmvp} and~\eqref{prob:cprmvp2}. For both approaches the stopping tolerance was set to $10^{-8}$: the relative MIP gap tolerance for Gurobi, and the tolerance $\epsilon$ in the termination test $ub^*-lb\leq\epsilon$ for the branch-and-bound algorithm. To ensure a fair comparison, the number of threads used by Gurobi was restricted to one and its time limit was set to 12 hours. In the benchmarking experiments reported in Tables~\ref{tab:results1}--\ref{tab:results2}, the return of the risk-free asset $r_c$ is set to $0.0002$, the target return $\bar{r}$ to $5\%$ above $r_c$, and $T=1$ for~\eqref{prob:cprmvp2}. Our support-wise analysis relies on Assumption~\ref{assump:feas}, which guarantees feasibility of every nonempty support-restricted subproblem of~\eqref{prob:cprmvp} and excludes degenerate support solutions of~\eqref{prob:cprmvp2}. We fix $\gamma=0.001$ and retain, before running either method, the assets satisfying $\abs{\ccr[i]}/\sqrt{d_i[i]}\leq\gamma$. The same retained investment universe is used by BnB and Gurobi, and the asset counts in Table~\ref{tab:datasets} report the resulting dimensions.

Tables~\ref{tab:results1}--\ref{tab:results2} summarize the computational performance of the proposed BnB algorithm and Gurobi. The columns labeled ``BnB CPU Time (s)'' and ``Gurobi CPU Time (s)'' report total CPU times in seconds. Entries marked with a dash (``-'') indicate that Gurobi did not certify optimality within 12 hours; its incumbent at termination is used to compute the reported sparsity and error. The ``Solution Sparsity'' column denotes the number of nonzero entries in the reported solution; when a single value is listed, both methods return solutions with the same sparsity, whereas when two values are listed, the first corresponds to BnB and the second to Gurobi. The column ``Drop Rate'' represents the proportion of assets removed during the warm-start phase, with 0 indicating that no elimination is performed. The column labeled ``Error (\%)'' reports the relative objective error $100(v_{\mathrm{BnB}}-v_{\mathrm{G}})/\abs{v_{\mathrm{G}}}$, where $v_{\mathrm{BnB}}$ and $v_{\mathrm{G}}$ denote the objective values of the solutions returned by BnB and by Gurobi at termination, respectively. For~\eqref{prob:cprmvp2}, the optimal value is sometimes negative, which is why the denominator carries an absolute value. Since Gurobi is run under a fixed time limit, this value should be interpreted relative to its best available feasible solution when optimality is not certified. A negative value therefore indicates that our BnB method obtained a strictly better objective than the solution Gurobi returned at the time limit. When Gurobi terminates before the time limit, it certifies global optimality up to its gap and feasibility tolerances; in that case, no method can produce a better objective value for the same instance beyond these tolerances. Accordingly, our method is not intended to improve upon such solutions, but rather to deliver high-quality solutions within reasonable computation times, particularly for larger instances. Finally, the column ``Node Reduction (\%)'' reports the percentage reduction in the number of nodes explored by the branch-and-bound algorithm due to the additional pruning step of Proposition~\ref{prop:pruning}, relative to the standard enumeration scheme of~\cite{SenAkkayaPinar2025}; a larger value indicates that a greater portion of the search tree is eliminated.

Table~\ref{tab:results1} presents the corresponding results for~\eqref{prob:cprmvp} across datasets of varying sizes. On the small-scale instances such as ItalianBonds, ETF, DowJones, and EuroStoxx50, both methods return identical solutions with no elimination (drop rate $0$); although the runtimes are small in absolute terms, BnB is already consistently faster than Gurobi. As the problem size increases (FTSE100 and NASDAQ100), the impact of the warm-start heuristic becomes more evident: by eliminating poorly performing assets it substantially reduces the problem size, and the returned solutions remain optimal (zero error) across all drop rates, with higher drop rates yielding the largest speedups and smaller drop rates increasing computation time. On FTSE100, Gurobi requires long computation times or fails to finish within the 12-hour limit, whereas BnB returns the same optimal solutions in seconds to a few minutes at higher drop rates, its runtime growing as the drop rate decreases. A similar pattern holds on NASDAQ100, where BnB outperforms Gurobi across all values of $\beta$ and drop rates while matching its optimal objective. For the largest dataset, S\&P500, Gurobi does not certify a solution within the time limit, while BnB returns solutions much faster but with a nonzero objective error (roughly $6\%$ to $14\%$); here smaller drop rates reduce the error at the expense of additional computation time. Overall, the results suggest that the warm-start elimination strategy removes poorly performing assets and reduces the problem size substantially, while preserving optimality on all instances except the largest, where a controllable trade-off between computation time and solution quality remains. Finally, the additional pruning step substantially reduces the search effort on every instance, removing between  20\%  and 80\% of the nodes explored by the enumeration scheme of~\cite{SenAkkayaPinar2025}.

Table~\ref{tab:results2} reports the results for~\eqref{prob:cprmvp2}, which follow a pattern comparable to that of~\eqref{prob:cprmvp}. On the smaller datasets (ItalianBonds, ETF, DowJones, and EuroStoxx50) no assets are eliminated and the two methods return nearly identical portfolios, with BnB several times faster than Gurobi. For EuroStoxx50 with $\beta=5\times10^{-4}$, the reported $-0.74\%$ error is due to numerical rounding rather than a genuine improvement over Gurobi. On the larger instances (FTSE100 and NASDAQ100), Gurobi no longer certifies optimality within the 12-hour limit, and the warm-start heuristic becomes decisive: higher drop rates yield larger speedups at a small cost in accuracy, whereas lower drop rates reduce the error at the expense of additional computation time. On FTSE100 the error even becomes negative at the lower drop rates, where BnB improves upon the solution Gurobi returns at the time limit (for $\beta=10^{-3}$, from $1.84\%$ at drop rate $0.8$ to $-1.39\%$ at $0.6$), while on NASDAQ100 it stays below $1.3\%$. For the largest dataset, S\&P500, BnB is again substantially faster, although the accuracy gap widens to between roughly $11\%$ and $24\%$ and narrows as the drop rate is reduced. Overall, the warm-start elimination substantially reduces the computation time of the larger instances and usually produces good-quality solutions, at the cost of some accuracy loss when the elimination is too aggressive. The node reduction is more modest here, ranging from 0\% to roughly 42\%, and it is largest on the instances whose optimal solutions are sparsest. This is not surprising because the pruning condition $lb^R+\beta\geq ub^*-\epsilon$ is harder to satisfy when the portfolios have larger supports. Indeed, the reported solutions of~\eqref{prob:cprmvp2} contain $5$--$271$ nonzero components, compared with only $1$--$8$ for~\eqref{prob:cprmvp}.
\begin{table}[!htbp]
\centering
\fontsize{10}{12}\selectfont
\caption{Dataset characteristics}
\label{tab:datasets}
\begin{tabular}{lccc}
\hline
\textbf{Index} & \textbf{Num. of assets} & \textbf{Days} & \textbf{Time interval} \\
\hline
ItalianBonds      & 11 & 1564  & 1/2013--12/2018 \\
ETF & 24 & 1042 &  1/2015--12/2018 \\
DowJones  & 28  & 4276 & 10/2006--2/2023 \\
EuroStoxx50      & 46 & 4276  & 10/2006--2/2023 \\
FTSE100      & 82 & 4276  & 10/2006--2/2023 \\
NASDAQ100      & 70 & 4276  & 10/2006--2/2023 \\
S\&P500      & 420 & 4276  & 10/2006--2/2023 \\
\hline
\end{tabular}
\end{table}
\begin{table}[!htbp]
\centering
\fontsize{10}{12}\selectfont
\setlength{\tabcolsep}{3.5pt}
\renewcommand{\arraystretch}{0.95}
\caption{Computational results comparison between BnB and Gurobi for~\eqref{prob:cprmvp}.}
\label{tab:results1}
\begin{tabular}{lcccrrcr}
\hline
\textbf{Dataset} & $\boldsymbol{\beta}$ &
\makecell{\textbf{Solution} \\ \textbf{Sparsity}} &
\makecell{\textbf{Drop} \\ \textbf{Rate}} &
\makecell{\textbf{BnB CPU} \\ \textbf{Time (s)}} &
\makecell{\textbf{Gurobi CPU} \\ \textbf{Time (s)}} &
\makecell{\textbf{Error} \\ \textbf{(\%)}} &
\makecell{\textbf{Node} \\ \textbf{Reduction (\%)}} \\
\hline
ItalianBonds & 5E-5 & 1 & 0 & 0.01 & 0.05 & 0 & 35.35 \\
        & 1E-5 & 2 & 0 & 0.03 & 0.06 & 0 & 20.22 \\
ETF     & 1E-4 & 2 & 0 & 0.06 & 0.27 & 0 & 57.43 \\
        & 5E-5 & 3 & 0 & 0.12 & 0.57 & 0 & 48.02 \\
DowJones & 5E-4 & 1 & 0 & 0.01 & 0.26 & 0 & 65.45 \\
        & 1E-4 & 2 & 0 & 0.02 & 0.42 & 0 & 71.30 \\
EuroStoxx50 & 5E-4 & 1 & 0 & 0.02 & 0.67 & 0 & 79.78 \\
        & 1E-4 & 3 & 0 & 0.52 & 3.28 & 0 & 73.20 \\
NASDAQ100 & 1E-4 & 2 & 0.5 & 1.20 & 56.24 & 0 & 70.30 \\
        & 1E-4 & 2 & 0.4 & 2.87 & 56.24 & 0 & 75.07 \\
        & 5E-5 & 3 & 0.5 & 5.44 & 264.82 & 0 & 61.92 \\
        & 5E-5 & 3 & 0.4 & 17.23 & 264.82 & 0 & 68.95 \\
        & 1E-5 & 8 & 0.5 & 132.76 & 35365.32 & 0 & 42.39 \\
        & 1E-5 & 8 & 0.4 & 1261.74 & 35365.32 & 0 & 44.01 \\
FTSE100 & 5E-5 & 3 & 0.7 & 0.24 & 365.39 & 0 & 47.97 \\
        & 5E-5 & 3 & 0.6 & 1.26 & 365.39 & 0 & 54.00 \\
        & 5E-5 & 3 & 0.5 & 3.02 & 365.39 & 0 & 57.60 \\
        & 1E-5 & 8 & 0.7 & 4.40 & - & 0 & 38.91 \\
        & 1E-5 & 8 & 0.6 & 66.27 & - & 0 & 41.48 \\
        & 1E-5 & 8 & 0.5 & 386.43 & - & 0 & 48.11 \\
S\&P500 & 5E-5 & 3 & 0.9 & 37.64 & - & 14.03 & 67.85 \\
        & 5E-5 & 3-8 & 0.85 & 377.10 & - & 6.68 & 74.52 \\
        & 1E-5 & 8-3 & 0.9 & 6412.85 & - & 5.82 & 49.17 \\
\hline
\end{tabular}
\end{table}
\begin{table}[!htbp]
\centering
\fontsize{10}{12}\selectfont
\setlength{\tabcolsep}{3.5pt}
\renewcommand{\arraystretch}{0.95}
\caption{Computational results comparison between BnB and Gurobi for~\eqref{prob:cprmvp2}.}
\label{tab:results2}
\begin{tabular}{lcccrrcr}
\hline
\textbf{Dataset} & $\boldsymbol{\beta}$ &
\makecell{\textbf{Solution} \\ \textbf{Sparsity}} &
\makecell{\textbf{Drop} \\ \textbf{Rate}} &
\makecell{\textbf{BnB CPU} \\ \textbf{Time (s)}} &
\makecell{\textbf{Gurobi CPU} \\ \textbf{Time (s)}} &
\makecell{\textbf{Error} \\ \textbf{(\%)}} &
\makecell{\textbf{Node} \\ \textbf{Reduction (\%)}} \\
\hline
ItalianBonds & 5E-3 & 5 & 0 & 0.04 & 0.11 & 0 & 0.00 \\
        & 1E-3 & 6 & 0 & 0.03 & 0.20 & 0 & 0.00 \\
ETF & 1E-3 & 14 & 0 & 0.62 & 5.98 & 0 & 15.04 \\
        & 5E-4 & 17 & 0 & 0.28 & 2.81 & 0 & 11.20 \\
DowJones & 1E-3 & 9 & 0 & 0.32 & 7.46 & 0 & 41.61 \\
        & 5E-4 & 12 & 0 & 0.34 & 6.03 & 0 & 30.53 \\
EuroStoxx50 & 5E-4 & 14-15 & 0 & 20.50 & 171.00 & -0.74 & 30.42 \\
        & 1E-4 & 26 & 0 & 15.16 & 138.85 & 0 & 19.52 \\
NASDAQ100 & 1E-3 & 17 & 0.5 & 538.27 & - & 0.48 & 29.25 \\
        & 5E-4 & 27-31 & 0.5 & 66.36 & - & 1.22 & 17.98 \\
        & 5E-4 & 30-31 & 0.4 & 1283.38 & - & 0.59 & 20.25 \\
FTSE100 & 1E-3 & 11-17 & 0.8 & 0.07 & - & 1.84 & 3.74 \\
        & 1E-3 & 14-17 & 0.7 & 2.82 & - & -0.10 & 13.78 \\
        & 1E-3 & 16-17 & 0.6 & 154.19 & - & -1.39 & 19.32 \\
        & 5E-4 & 13-28 & 0.8 & 0.02 & - & 6.45 & 0.00 \\
        & 5E-4 & 21-28 & 0.7 & 1.06 & - & 2.44 & 4.30 \\
        & 5E-4 & 23-28 & 0.6 & 20.56 & - & 0.61 & 10.46 \\
S\&P500 & 1E-3 & 22-37 & 0.9 & 657.65 & - & 14.41 & 25.49 \\
        & 5E-4 & 31-73 & 0.9 & 169.11 & - & 22.62 & 12.67 \\
        & 1E-4 & 58-222 & 0.85 & 8.55 & - & 24.25 & 0.27 \\
        & 1E-4 & 78-222 & 0.8 & 564.04 & - & 16.67 & 0.00 \\
        & 5E-5 & 80-271 & 0.8 & 9.79 & - & 19.90 & 0.03 \\
        & 5E-5 & 120-271 & 0.7 & 5448.01 & - & 10.94 & 0.00 \\
\hline
\end{tabular}
\end{table}
\subsection{Out-of-Sample Performance}\label{sec:oos}
We compare the robust ($\gamma>0$) and purely sparse ($\gamma=0$) models out of sample. We run a rolling-window backtest on the $100$ size and book-to-market-sorted assets from the Kenneth French Data Library\footnote{Data available at \url{https://mba.tuck.dartmouth.edu/pages/faculty/ken.french/data_library.html}.}
 (daily value-weighted returns, {July 2004 to June 2026}). At each rebalancing date, we estimate $D$ and $\ccr$ from an in-sample window of $252$ trading days and evaluate the resulting allocation out of sample over the next $21$ trading days. Since the estimates are recomputed at each date, Assumption~\ref{assump:feas} need not hold for the whole universe: assets with $\abs{\ccr[i]}/\sqrt{d_i[i]}\leq\gamma$ are removed from the candidate set before solving, and rebalancing dates at which no asset survives this screening are skipped. The same screening and the same set of rebalancing dates are used for the robust and the sparse model, so that the two are compared on identical out-of-sample periods. This leaves $215$ rebalances over the sample period. The risk-free return is set to $r_c=5\times10^{-5}$ and the target return to $\bar r = 1.05\,r_c$. Both models are solved by the branch-and-bound algorithm using the warm-start elimination of Section~\ref{sec:BnBwarmstart} at a drop rate of $0.3$.

We report the annualized out-of-sample Sharpe ratio, the annualized out-of-sample mean excess return, and the average number of selected assets $\abs{\sigma}$.
\begin{table}[!htbp]
\centering
\caption{{Out-of-sample robust ($\gamma>0$) versus sparse ($\gamma=0$)
performance on the Fama--French 100 Size$\times$Book-to-Market portfolios. Mean excess
returns are annualized and expressed in percent; $\abs{\sigma}$ is the average number of
selected assets.}}
\label{tab:oos}
{\setlength{\tabcolsep}{5pt}
\begin{tabular}{cc cc cc cc}
\hline
 & & \multicolumn{2}{c}{\textbf{Sharpe}} & \multicolumn{2}{c}{\textbf{Mean return (\%)}}
   & \multicolumn{2}{c}{$\boldsymbol{\abs{\sigma}}$} \\
$\boldsymbol{\beta}$ & $\boldsymbol{\gamma}$ & Sparse & Robust & Sparse & Robust
   & Sparse & Robust \\
\hline
$10^{-6}$ & $0.10$ & $0.3160$ & $0.4716$ & $0.0102$ & $0.0945$ & $1.06$ & $2.95$ \\
$10^{-6}$ & $0.15$ & $0.4505$ & $0.5721$ & $0.0111$ & $0.2905$ & $1.00$ & $2.94$ \\
$5\times10^{-7}$ & $0.10$ & $0.4308$ & $0.4312$ & $0.0133$ & $0.0862$ & $1.32$ & $3.75$ \\
$5\times10^{-7}$ & $0.15$ & $0.4987$ & $0.5704$ & $0.0119$ & $0.2896$ & $1.06$ & $3.38$ \\
\hline
\end{tabular}
}
\end{table}
Across the tested configurations the robust model attains higher out-of-sample Sharpe ratios and mean excess returns than the sparse model in most settings. We do not claim that the robust model dominates the sparse one in general. Nonetheless, these results suggest that the robustness term can improve out-of-sample performance, which makes robust sparsity a worthwhile alternative to purely sparse portfolio selection.
\section{Conclusion}\label{conclusion}
In this paper, we considered a mean–variance portfolio selection problem that accounts for uncertainty in expected returns through an ellipsoidal uncertainty set, while at the same time promoting sparsity via an $\ell_0$-penalty. Bringing these two aspects together leads to a problem that is both nonconvex and discontinuous, and therefore inherently difficult to solve. To better understand this structure, we carried out a detailed analysis of both local and global minimizers for the two formulations studied, namely the robust risk minimization and robust return maximization models. In doing so, we clarified how local minimizers relate to support-restricted subproblems, established existence results for global solutions, and derived explicit lower and upper bounds on their components.

These structural results guided the design of a tailored branch-and-bound algorithm. In particular, the bounds we obtained proved useful not only for pruning the search space, but also for constructing effective warm-start solutions. This combination plays an important role in keeping the computational effort manageable, even though the underlying problem is combinatorial due to the $\ell_0$ term.

Our computational study on real financial data suggests that the proposed approach performs competitively against general-purpose mixed-integer second-order cone programming solvers, and in many cases achieves better performance in terms of running time without sacrificing solution quality. At the same time, the portfolios obtained reflect the intended balance: they are both robust to estimation errors and sparse enough to be practically implementable.

Overall, the paper offers a unified perspective on combining robustness and exact sparsity in portfolio optimization under ellipsoidal uncertainty. There are several natural directions for future work, including the use of different uncertainty sets, the incorporation of additional practical constraints such as transaction costs or turnover limits, and the extension to alternative risk measures within the same framework.

\section*{Acknowledgments}
Buse \c{S}en would like to acknowledge support as a part of NCCR Automation, a National Centre of Competence in Research, funded by the Swiss National Science Foundation (grant number 51NF40\_225155).

\bibliographystyle{plainnat} 
\bibliography{biblio}
\end{document}